\documentclass[11pt,a4paper]{amsart}
\usepackage[utf8]{inputenc}
\usepackage[english]{babel}
\usepackage{amsmath}
\usepackage{amssymb}
\usepackage{amsfonts} 
\usepackage{amsthm}
\usepackage{mathrsfs}
\usepackage{graphicx}
\usepackage{enumitem}
\setlist[itemize]{leftmargin=*}
\usepackage{xcolor}
\usepackage{url}
\usepackage{accents}
\usepackage{bm}
\usepackage[all]{xy}
\usepackage{hyperref}
\usepackage{caption}
\usepackage{subcaption}
\usepackage{booktabs}
\usepackage{mathtools}
\usepackage{etoolbox}
\usepackage{overpic}
\usepackage{xspace}
\usepackage{comment}
\usepackage{chngcntr}
\makeatletter
\let\ams@starttoc\@starttoc
\makeatother
\usepackage[parfill]{parskip}
\makeatletter
\let\@starttoc\ams@starttoc
\patchcmd{\@starttoc}{\makeatletter}{\makeatletter\parskip\z@}{}{}
\makeatother

\makeatletter
\newcommand{\boxedred}[1]{\textcolor{red}{\fbox{\normalcolor\m@th$#1$}}}
\makeatother
\makeatletter
\newcommand{\eqlabel}[1]{\refstepcounter{equation}\textup{\tagform@{\theequation}}\label{#1}}
\makeatother

\newcommand{\N}{\mathbb{N}}
\newcommand{\Z}{\mathbb{Z}}

\newcommand{\R}{\mathbb{R}}
\newcommand{\C}{\mathbb{C}}

\newcommand\restr[2]{
		#1
		\vphantom{\big|}
		\Big|_{#2}
}

\newcommand{\de}{\partial}

\newcommand{\interna}[1]{\accentset{\circ}{#1}}

\newcommand{\mz}{\frac{1}{2}}

\newcommand{\ang}[1]{\left\langle#1\right\rangle}
\newcommand{\uno}{\bm{1}}

\newcommand{\nin}{\not\in}

\newcommand{\weakstarto}{\stackrel{*}{\rightharpoonup}}
\newcommand{\inject}{\hookrightarrow}

\DeclareMathOperator{\sgn}{sgn}

\DeclareMathOperator{\dist}{dist}

\newtheoremstyle{boldstyle}
{\glueexpr\parskip*2\relax}                    
{\parskip}                    
{\upshape}                   
{}                           
{\bfseries}                   
{.}                          
{.2em}                       
{}  
\theoremstyle{boldstyle}

\newtheorem{definition}{Definition}[section]
\newtheorem{thm}[definition]{Theorem}
\newtheorem{lemmaen}[definition]{Lemma}
\newtheorem{corollary}[definition]{Corollary}
\newtheorem{proposition}[definition]{Proposition}

\newtheorem{rmk}[definition]{Remark}

\newtheorem*{definition*}{Definition}
\newtheorem*{thm*}{Theorem}
\newtheorem*{lemmaen*}{Lemma}
\newtheorem*{corollary*}{Corollary}
\newtheorem*{proposition*}{Proposition}
\newtheorem*{notazen*}{Notation}
\newtheorem*{rmk*}{Remark}
\newtheorem*{example*}{Example}
\newtheorem*{claim*}{Claim}
\newtheorem*{ack}{Acknowledgements}
\newcommand{\obar}[1]{\overline{#1}}

\newcommand{\set}[1]{\left\{#1\right\}}

\newcommand{\pa}[1]{\left(#1\right)}
\newcommand{\bra}[1]{\left[#1\right]}
\newcommand{\abs}[1]{\left|#1\right|}
\newcommand{\babs}[1]{\big|#1\big|}

\newcommand{\norm}[1]{\left\|#1\right\|}
\newcommand{\bnorm}[1]{\big\|#1\big\|}

\newcommand{\iA}{\accentset{\circ}{A}}
\newcommand{\env}{\R^m}
\newcommand{\dvg}{d\text{vol}_g}
\renewcommand{\N}{\mathcal{N}}

\newcommand{\tu}{{\tilde u}}
\newcommand{\obt}{{\obar{t}}}
\newcommand{\teich}{Teichm\"uller\xspace}
\newcommand{\tch}{\mathcal{T}}
\newcommand{\ls}{\lesssim}
\newcommand{\rz}{\mathcal{R}}
\newcommand{\ql}{(-\Delta)^{1/4}}
\newcommand{\mql}{(-\Delta)^{-1/4}}
\newcommand{\ml}{(-\Delta)^{1/2}}

\newcommand{\hardy}{{\mathcal{H}^1}}
\newcommand{\hmz}{{\dot H^{1/2}}}
\newcommand{\hmzb}{{\dot H^{1/2}\cap L^\infty}}
\newcommand{\hmmz}{{\dot H^{-1/2}}}
\newcommand{\subman}{\mathcal{N}}
\renewcommand{\hat}{\widehat}
\renewcommand{\tilde}{\widetilde}
\renewcommand{\bar}{\overline}
\usepackage[a4paper,left=2.25cm,right=2.25cm,top=3.5cm,bottom=3.5cm]{geometry}
\usepackage{scalerel}
\usepackage[usestackEOL]{stackengine}
\usepackage[stretch=30,shrink=30]{microtype}
\allowdisplaybreaks

\DeclareFontFamily{U}{mathx}{\hyphenchar\font45}
\DeclareFontShape{U}{mathx}{m}{n}{
	<5> <6> <7> <8> <9> <10>
	<10.95> <12> <14.4> <17.28> <20.74> <24.88>
	mathx10
}{}
\DeclareSymbolFont{mathx}{U}{mathx}{m}{n}
\DeclareFontSubstitution{U}{mathx}{m}{n}
\DeclareMathAccent{\widecheck}{0}{mathx}{"71}

\def\media{
	\,\ThisStyle{\ensurestackMath{%
			\stackinset{c}{.2\LMpt}{c}{.5\LMpt}{\SavedStyle-}{\SavedStyle\phantom{\int}}}%
		\setbox0=\hbox{$\SavedStyle\int\,$}\kern-\wd0}\int
}

\usepackage[backend=biber,sorting=anyt]{biblatex}

\usepackage{csquotes}

\begin{document}
	
\title[Free boundary minimal surfaces: a nonlocal approach]{Free boundary minimal surfaces: \\ a nonlocal approach}
\author{Francesca Da Lio}
\thanks{Department of Mathematics, ETH Z\"urich, R\"amistrasse 101, 8092 Z\"urich, Switzerland.}
\author{Alessandro Pigati}

\begin{abstract}
Given a $C^k$-smooth closed embedded manifold $\N\subset{\mathbb R}^m$, with $k\ge 2$, and a compact connected smooth Riemannian surface $(S,g)$ with $\partial S\neq\emptyset$, we  consider {\em   $\mz$-harmonic maps} $u\in H^{1/2}(\de S,\N)$. These maps are  critical points of the nonlocal energy
\begin{equation}\label{energyab} E(f;g):=\int_S\abs{\nabla\tilde u}^2\,d\text{vol}_g,\end{equation}
where $\tilde u$ is the harmonic extension of $u$ in $S$. We express the energy \eqref{energyab} as a sum of the $\mz$-energies at each boundary component of $\partial S$ (suitably identified with the circle $\mathcal S^1$), plus a quadratic term which is continuous in the $H^s(\mathcal S^1)$ topology, for any $s\in\R$. We show the $C^{k-1,\delta}$ regularity of $\mz$-harmonic maps. We also establish a connection between free boundary minimal surfaces and critical points of $E$ with respect to variations of the pair $(f,g)$, in terms of the Teichm\"uller space of $S$.
\end{abstract}

\maketitle
\tableofcontents

\section{Introduction}

Let $(S,g)$ be a connected $C^\infty$-smooth surface with nonempty boundary $\de S$, equipped with a smooth metric $g$ ($S$ is not necessarily oriented) and let $\N\subset\env$ be  an embedded closed (i.e. compact without boundary) $C^2$-smooth submanifold.
\par We set $$H^{1/2}(\de S,\N):=\set{f\in H^{1/2}(\de S,\env):f(x)\in\N\text{ for a.e. }x}.$$
	Given a map $f\in H^{1/2}(\de S,\N)$, we define the $\mz$-\emph{energy} of $f$ to be
	\begin{equation}\label{energy} E(f;g):=\int_S\babs{\nabla\tilde f}^2\,d\text{vol}_g,\end{equation}
	  Here $\tilde{f}$ denotes the harmonic extension of $f$, i.e. the unique harmonic map $\tilde{f}\in H^1(S,\env)$ such that $\restr{\tilde{f}}{\de S}=f$.
 We observe that $E(f,g)$ depends only on the conformal class of $g$.

\begin{definition}\label{halfhar}
	A map $u\in H^{1/2}(\de S,\N)$
	is called $\mz$-\emph{harmonic} if $u$ is a critical point for the $\mz$-energy $E=E(\cdot;g)$, in the following sense: for any $\phi\in C^\infty(\de S,\env)$ we have
\begin{equation}\label{halfharmeq}
\left.\frac{d}{dt}E(\Pi(u+t\phi))\right|_{t=0}=0,\end{equation}
	where  $\Pi:\mathcal{U}\to\N$ is any fixed $C^2$ projection, defined on some neighborhood $\mathcal{U}$ of $\N$.
\end{definition}
Definition \ref{halfhar} extends the one introduced for the first time in \cite{DLR09} in the case $S=\mathbb{D}$ or in the noncompact case $S=\mathbb{H}$ ($\mathbb{D}$ and $\mathbb{H}$ being the unit disk and the upper half-plane, respectively).
 One can check that $\Pi(u+t\phi)=u+tv+o(t)$  in $H^{1/2}(\de S,\R^m)$ as $t\to 0$  where
 $v:=d\Pi(u)[\phi]$ and therefore\footnote{The normal derivative $\frac{\de\tu}{\de\nu}\in H^{-1/2}(\de S,\R^m)$ is defined precisely by asking that, for any $v\in H^{1/2}(\de S,\R^m)$,
 	\[ \int_S\ang{\nabla\tu;\nabla\tilde v}\,\dvg=\int_{\de S}\frac{\de\tu}{\de\nu}\cdot v\,\dvg. \]}
	\[ \left.\frac{d}{dt}E(\Pi(u+t\phi))\right|_{t=0}=2\int_S\ang{\nabla\tu;\nabla\tilde v}\,\dvg=2\int_{\de S}d\Pi(u)[\phi]\cdot\frac{\de\tu}{\de\nu}\,\dvg.\]
	By a standard density argument, $u$ is $\mz$-harmonic if and only if 
	 \begin{equation}\label{eq:hhchar1}\int_{\de S}d\Pi(u)[\phi]\cdot\frac{\de\tu}{\de\nu}\,\dvg=0,\end{equation}
	  for any $\phi\in L^\infty\cap H^{1/2}$ (which is a Banach algebra), which   is in turn equivalent to ask
	 \begin{equation}\label{eq:hhchar2} \int_{\de S}\frac{\de\tu}{\de\nu}\cdot v\,\dvg=0 \end{equation}
	 for any $v\in L^\infty\cap H^{1/2}(\de S,\R^m)$ satisfying $v\in T_u\N$ a.e. In particular, the definition is independent of the choice of $\Pi$.

Let $P^T(\xi)$ denote the orthogonal projection onto the tangent space $T_\xi\N$, for $\xi\in\N$, and observe that $P^T\in C^1(\N,\R^{m\times m})$. In the paper we will also call $P^N:=I-P^T$ the projection onto the normal space. The same argument showing the equivalence of \eqref{eq:hhchar1} and \eqref{eq:hhchar2} proves that one can replace $d\Pi$ with $P^T$ in \eqref{eq:hhchar1} (notice that, on $\N$, $P^T$ is the differential of the nearest point projection, canonically defined near $\N$, but we cannot use this projection in \eqref{halfharmeq} as it is merely $C^1$).
	  Hence, 	 \begin{equation}\label{eq:hhfchar} \begin{split} u\text{ is }\mz\text{-harmonic}\quad&\Leftrightarrow\quad\int_{\de S}\frac{\de\tu}{\de\nu}\cdot P^T(u)v\,\dvg=0,\quad\forall v\in L^\infty\cap H^{1/2}(\de S,\R^m) \\
	 &\Leftrightarrow\quad P^T(u)\frac{\de\tu}{\de\nu}=0\quad\text{in }\mathcal{D}'(\de S). \end{split} \end{equation}
	 
Solutions to the last equation are of special geometric interest because they are strictly connected with the so-called  {\em free boundary minimal surfaces}, in the following sense.

\begin{definition}
 \label{freedef}
 We say that a map $\tilde u\in C^2(S,R^m)$  is a
 \emph{free boundary (branched) minimal immersion}  with supporting manifold $\N$  if it is a harmonic map which is also conformal (with the possible exception of finitely many points where $d\tilde u$ vanishes) and meets $\N$ orthogonally, i.e.
  \[
 P^T(u)\frac{\de\tu}{\de\nu}=0\quad\text{on }\de S.
\] 
  \end{definition}
In the case  $S=\mathbb{D}$   the following connection between  $\mz$-harmonic   maps $u:{\mathcal{S}}^1\to\N$ and the free boundary discs is now a well-known fact
(see e.g. \cite{DaL15, MS15,DaL17} and Remark \ref{diskconf}).

 \begin{proposition}
\label{pr-I.1} 
The harmonic extension $\tilde u$ of a $\mz$-harmonic map $u\in{H}^{1/2}(\mathcal S^1,\N)$ is conformal. Geometrically, this means that $u$ is the boundary of a free boundary (branched) minimal disc.
\end{proposition}
We point out that Proposition \ref{pr-I.1}   has been at the origin of the study of $1/2$-harmonic maps.

In  this paper we are going to  investigate the regularity of     $\mz$-harmonic maps $u\in H^{1/2}(\de S,\N)$. Besides showing the H\"older  continuity of such maps, we will illustrate  how to bootstrap to higher regularity. Precisely   we will show the following.
 \begin{thm}\label{regul}
Let $\N\subset\R^m$ be a  $C^k$-smooth closed embedded manifold, with $k\ge 2$, and let $u\in {H}^{1/2}(\de S,\N)$ be $\mz$-harmonic. Then
\[ u\in
\bigcap_{0<\delta<1} C^{k-1,\delta}(\de S,\N). \]
In particular, if $\N$ is $C^\infty$ then $u\in C^\infty(\de S,\N)$.
\end{thm}
The proof of Theorem  \ref{regul}  is rather technical and we defer it  to  the appendix.

We point out that one of the key steps to prove the regularity  of  $\mz$-harmonic maps is 
the representation of the energy $E(f;g)$ as a sum of the fractional $\mz$-energies at each boundary component (according to a suitable identification with ${\mathcal{S}}^1$), plus a quadratic term which is continuous in the $H^s$ topology, for any $s\in\R$. The identification of the energy of $\tilde f$ with a fractional energy on the boundary   in the case of flat disk $\mathbb{D}$ is a well-known fact.

In the model case where $S=A_t:=\bar B_t\setminus B_1$, $t>1$, we have  the following decomposition.
\begin{lemmaen}\label{anndecInt}
Let $a,b\in {{H}}^{1/2}({\mathcal{S}}^1,\env)$ and define $f\in H^{1/2}(\de A_t,\env)$ by setting  $f(e^{i\theta})=a(e^{i\theta})$, $f(te^{i\theta})=b(e^{i\theta})$. Then the Dirichlet energy of the harmonic extension $\tilde{f}\in H^1(A_t,\env)$ is given by
\begin{equation}\label{eq:annenergyint} \begin{split} \frac{1}{2\pi}\int_{A_t}\babs{\nabla\tilde f}^2 
=&\sum_n\abs{n}(\abs{a_n}^2+\abs{b_n}^2)+\frac{\abs{b_0-a_0}^2}{\log t} \\
&+\sum_{n>0}n\pa{\frac{4}{t^{2n}-1}(\abs{a_n}^2+\abs{b_n}^2)-\frac{8t^n}{t^{2n}-1}\Re(a_n\cdot\obar{b_n})} \\
 =&  \frac{1}{2\pi}[\|{(-\Delta)^{1/4}a}\|_{L^2({\mathcal{S}}^1)}^2+\|{(-\Delta)^{1/4}b}\|_{L^2({\mathcal{S}}^1)}^2+{\mathcal{B}_t}((a,b),(a,b))], \end{split}\end{equation}
 where ${\mathcal{B}}_t:\mathcal{D}'({\mathcal{S}}^1,\env)^2\times\mathcal{D}'({\mathcal{S}}^1,\env)^2\to\R$ is a compact symmetric bilinear functional.
\end{lemmaen}

By using the decomposition  \eqref{eq:annenergyint} we succeed in rewriting condition \eqref{eq:hhfchar}   
in the form of a {nonlocal linear Schr\"odinger system with an antisymmetric potential}, as it has been done in \cite{DLR11,MS17,DLS17} in the case of the flat disk.

We will also show that the \emph{conformality}\footnote{\emph{Conformality} will mean \emph{weak conformality}, i.e. at every point $d\tilde u$ either is a linear conformal map or vanishes.} of the harmonic extension $\tilde u$ is equivalent to criticality of $E$ with respect to variations of the \emph{conformal class}   of $S$.
For instance, if $S$ is diffeomorphic to an annulus, then up to a conformal diffeomorphism we can assume that $(S,g)=(A_t,g_{\R^2})$ for some $t>0$ (see Theorem \ref{unifann}). In this case a variation of the conformal class corresponds to a variation of the parameter $t$.

\begin{thm}\label{confannIntr}
	Let $a,b\in {{H}}^{1/2}({\mathcal{S}}^1,\N)$ and define $u[t]\in{ {H}}^{1/2}(\de A_t,\N)$ by $u[t](e^{i\theta}):=a(e^{i\theta})$, $u[t](te^{i\theta}):=b(e^{i\theta})$.
	Assume that $u[\obar{t}]$ is $\mz$-harmonic for the annulus $(A_{\obar{t}},g_{\R^2})$. Then its harmonic extension is conformal if and only if \[ \restr{\frac{d}{dt}E_t(u[t])}{t=\obar{t}}=0, \]
	where $E_t$ is the $\mz$-energy for $A_t$.
\end{thm}

We will extend Theorem \ref{confannIntr} to the hyperbolic case  where $S$ is neither a disk nor an annulus: see Theorem \ref{hypconf}.  \par
In the interesting special case   where $\N $  is the boundary of a {\em convex} $C^\infty$-smooth domain $\Omega$, we also prove the following result.

\begin{corollary}
\label{cvx}
The harmonic extension $\tilde u$ defines a conformal (branched) free boundary minimal immersion $\tilde u\colon (\interna{S},\partial S)\to (\Omega,\de\Omega)$, with branch points \emph{only in} $\interna{S}$, if and only if
$u$ is a nontrivial critical point of $E(f;g)$ with respect to the pair $(f,g)$ (criticality with respect to $g$ is understood in the sense specified in Section \ref{confsec}).
\end{corollary}
 
 In view of the results in this paper, it would be interesting to study the flow version of the energy $E(f;g)$, where the evolution of the conformal class of $g$ would be given by
 the lack of conformality of $\tilde u$, in a similar way as for the Teichm\"uller harmonic map flow studied in \cite{RT16}.
 This would correspond to the Teichm\"uller $\mz$-harmonic  flow.
 
 This paper is organized as follows.
 \begin{itemize}
 	\item Section \ref{decEn} provides the decomposition of the $\mz$-energy \eqref{energy} in terms of nonlocal operators defined on $\partial S$; we also obtain a similar decomposition for the related Dirichlet-to-Neumann operator.
 	\item Section \ref{confsec} establishes the criterion for the conformality of the harmonic extension $\tilde f$, as well as Corollary \ref{cvx}.
 	\item In Section \ref{unifannsec} we show a well-known uniformization theorem for compact annuli, exhibiting a conformal equivalence which is smooth up to the boundary; this is needed for the construction made in Section \ref{decEn}.
 	\item Section \ref{functapp} collects the definitions and basic facts concerning all the functional spaces involved in the paper; in particular we show some useful results about the space $\hmz(\R)$.
 	\item In Section \ref{tcsec} we recall some fundamental three-term commutator estimates, which were first obtained in \cite{DLR09}, as well as a two-term commutator estimate  due to Coifman--Rochberg--Weiss from \cite{CRW76}.
 	\item Section \ref{holdereg} details the proof of the H\"older continuity of a $\mz$-harmonic map $u$ and uses localized versions of the integrability by compensation effects recalled in Section \ref{tcsec}.
 	\item In Section \ref{higherreg} we bootstrap the results of Section \ref{holdereg} to obtain higher regularity of $u$, i.e. Theorem \ref{regul}, exploiting another two-term commutator.
 \end{itemize}

\begin{ack}The authors would like to thank Tristan Rivi\`ere for
suggesting the investigation of  the problem and for the helpful discussions.\end{ack}
  
\section{Decomposition of the energy}\label{decEn}

The purpose of this section is to obtain the decomposition of the $\mz$-energy \eqref{energy} in terms of nonlocal operators defined on $\partial S$.\par
We will also show  that the so-called Dirichlet-to-Neumann  operator
\[ H^{1/2}(\de S,\R^m)\to H^{-1/2}(\de S,\R^m),\quad f\mapsto\frac{\de\tilde f}{\de\nu} \]
can be represented  as the sum of the usual fractional Laplacian at each boundary component and  a remainder $\mathcal{B}\colon \mathcal{D}'(\de S,\R^m)\to C^\infty(\de S,\R^m)$,   which represents a sort of interaction between the boundary data.

We will start from the model case of the flat annulus, where this decomposition is explicit.

\subsection{The case of an annular domain}

For a fixed $t>1$, let $A_t:=\bar{B}_t\setminus B_1\subset\C$ be the standard annulus with the euclidean metric.

Given  $f\in H^{1/2}(\de A_t,\env)$, we denote  
\[ a(e^{i\theta}):=f(e^{i\theta}),\ b(e^{i\theta}):=f(te^{i\theta})\in H^{1/2}({\mathcal{S}}^1,\env). \]
We use the notation $(a_n)_{n\in\Z}$ and $(b_n)_{n\in\Z}$ for the Fourier coefficients of the two functions, namely
\[ a_n:=\frac{1}{2\pi}\int_0^{2\pi}a(e^{i\theta})e^{-in\theta}\,d\theta, \quad b_n:=\frac{1}{2\pi}\int_0^{2\pi}b(e^{i\theta})e^{-in\theta}\,d\theta. \]

We observe that $\sum_{n\in\Z} 2\pi\abs{n}\abs{a_n}^2=\|{(-\Delta)^{1/4}a}\|_{L^2}^2$ and similarly for $b$.
\par
Given $(a,b),(c,d)\in \mathcal{D}'({\mathcal{S}}^1,\env)^2\times\mathcal{D}'({\mathcal{S}}^1,\env)^2$ we define the following bilinear operator:
\begin{align} \label{btdef}{\mathcal{B}_t}((a,b),(c,d))&:=2\pi\frac{(b_0-a_0)\cdot(d_0-c_0)}{\log t} \\
	&\quad+\sum_{n>0}\frac{8\pi n}{t^{2n}-1}\Re(a_n\cdot\obar{c_n}+b_n\cdot\obar{d_n}-t^n a_n\cdot\obar{d_n}-t^n b_n\cdot\obar{c_n}). \end{align}

\begin{lemmaen}\label{bt}	
${\mathcal{B}_t}$ is a sequentially continuous bilinear functional on $\mathcal{D}'({\mathcal{S}}^1,\env)^2\times \mathcal{D}'({\mathcal{S}}^1,\env)^2$.
\end{lemmaen}
\begin{proof}
	Assume $a,b,c,d\in H^s(\mathcal S^1,\env)$. Since $t>1$ we have
	\begin{align} \label{contB}\frac{\abs{{\mathcal{B}_t}((a,b),(c,d))}}{2\pi}\le&\frac{\abs{b_0-a_0}\abs{d_0-c_0}}{\log t}+\sum_{n>0}\frac{4n}{t^{2n}-1}(\abs{a_n}\abs{c_n}+\abs{b_n}\abs{d_n}) \nonumber\\
	&+\sum_{n>0}\frac{4nt^n}{t^{2n}-1}(\abs{a_n}\abs{d_n}+\abs{b_n}\abs{c_n}) \\
	\lesssim&\sum_{n\ge 0}(1+n^2)^s(\abs{a_n}\abs{c_n}+\abs{b_n}\abs{d_n}+\abs{a_n}\abs{d_n}+\abs{b_n}\abs{c_n}) \nonumber \\
	\le&\norm{a}_{H^s}\norm{c}_{H^s}+\norm{b}_{H^s}\norm{d}_{H^s}+\norm{a}_{H^s}\norm{d}_{H^s}+\norm{b}_{H^s}\norm{c}_{H^s},  \nonumber\end{align}
	thanks to the elementary estimate $\frac{nt^n}{t^{2n}-1}\ls n^{2s}$ and the Cauchy--Schwarz inequality (the implied constants depend of course on $s,t$).
	Since $\mathcal{D}'(\mathcal S^1)=\bigcup_{s\in\R}H^s(\mathcal S^1)$, we get in particular that $\mathcal B_t$ is a linear functional on $\mathcal{D}'({\mathcal{S}}^1,\env)^2\times \mathcal{D}'({\mathcal{S}}^1,\env)^2$.
	
	If $(((a_i,b_i),(c_i,d_i)))_{i\in\mathbb N}$ is a sequence converging to $((a,b),(c,d))$ in this space, by the uniform boundedness principle (applied to the Fr\'echet space $\mathcal{D}({\mathcal{S}}^1,\R^m)$) we deduce that the set $\set{a_i,b_i,c_i,d_i\mid i\in\mathbb N}$ is bounded in $H^{s+1}(\mathcal S^1)$, for some real $s$.
	
	By the compact embedding $H^{s+1}(\mathcal{S}^1)\inject H^s(\mathcal S^1)$, any subsequence admits a further subsequence converging in the $H^s({\mathcal{S}}^1,\env)^2\times H^s({\mathcal{S}}^1,\env)^2$, where we have already shown the continuity of ${\mathcal{B}_t}$. This shows that ${\mathcal{B}_t}((a_i,b_i),(c_i,d_i))\to {\mathcal{B}_t}((a,b),(c,d))$.
\end{proof}

\begin{lemmaen}\label{anndec}
	For any $f\in H^{1/2}(\de A_t,\env)$, the Dirichlet energy of its harmonic extension $\tilde{f}\in H^1(A_t,\env)$ is given by
\begin{equation}\label{eq:annenergy} \begin{split} \int_{A_t}\babs{\nabla\tilde f}^2 
=&2\pi\sum_n\abs{n}(\abs{a_n}^2+\abs{b_n}^2)+2\pi\frac{\abs{b_0-a_0}^2}{\log t} \\
&+2\pi\sum_{n>0}n\pa{\frac{4}{t^{2n}-1}(\abs{a_n}^2+\abs{b_n}^2)-\frac{8t^n}{t^{2n}-1}\Re(a_n\cdot\obar{b_n})} \\
 =&  \bnorm{(-\Delta)^{1/4}a}_{L^2({\mathcal{S}}^1)}^2+\bnorm{(-\Delta)^{1/4}b}_{L^2({\mathcal{S}}^1)}^2+{\mathcal{B}_t}((a,b),(a,b)). \end{split}\end{equation}
	\end{lemmaen}
\begin{proof} One can   check, e.g. by a density argument involving trigonometric polynomials, that the harmonic extension $\tilde f$ is given by
\begin{equation}\label{eq:annharmext} \tilde f(re^{i\theta})=a_0+\frac{b_0-a_0}{\log t}\log r+\sum_{n\neq 0} \frac{t^n b_n-a_n}{t^{2n}-1} r^n e^{in\theta}+\sum_{n\neq 0} \frac{t^{2n}a_n-t^n b_n}{t^{2n}-1} r^{-n}e^{in\theta}. \end{equation}
Calling
\[ \tilde c=\frac{b_0-a_0}{\log t},\quad c_n=\frac{t^n b_n-a_n}{t^{2n}-1},\quad c_n'=\frac{t^{2n}a_n-t^n b_n}{t^{2n}-1}, \]
we have
\begin{equation}\label{eq:anngrad} \begin{split} \frac{\de\tilde f}{\de r}(r,\theta)=&\,\tilde c r^{-1}+\sum_{n\neq 0}n(c_n r^{n-1}-c_n' r^{-n-1})e^{in\theta}, \\
\frac{1}{r}\frac{\de\tilde f}{\de\theta}=&\sum_{n\neq 0}in(c_n r^{n-1}+c_n'r^{-n-1})e^{in\theta}. \end{split} \end{equation}
Thus the Dirichlet energy of $\tilde{f}$ equals
\begin{align*} \int_{A_t}\babs{\nabla\tilde f}^2
=&2\pi\abs{\tilde c}^2\log t
+2\pi\sum_{n\neq 0}n\pa{\abs{c_n}^2(t^{2n}-1)-\abs{c_n'}^2(t^{-2n}-1)} \\
=&2\pi\frac{\abs{b_0-a_0}^2}{\log t}+2\pi\sum_{n\neq 0}\frac{n}{t^{2n}-1}\pa{\abs{t^n b_n-a_n}^2+\abs{t^n a_n-b_n}^2}. \end{align*}
Since $a_{-n}=\obar{a_n}$ and $b_{-n}=\obar{b_n}$, we deduce
\begin{align*} \frac{\int_{A_t}\babs{\nabla\tilde f}^2}{2\pi}
=&\frac{\abs{b_0-a_0}^2}{\log t}+\sum_{n\neq 0}\frac{n}{t^{2n}-1}\pa{(t^{2n}+1)(\abs{a_n}^2+\abs{b_n}^2)-4t^n\Re(a_n\cdot\obar{b_n})} \\
=&\frac{\abs{b_0-a_0}^2}{\log t}+\sum_{n\neq 0}\abs{n}(\abs{a_n}^2+\abs{b_n}^2) \\
&+\sum_{n>0}n\pa{\frac{4}{t^{2n}-1}(\abs{a_n}^2+\abs{b_n}^2)-\frac{8t^n}{t^{2n}-1}\Re(a_n\cdot\obar{b_n})}. \qedhere \end{align*}
\end{proof}
\begin{lemmaen}
 	Let $A_t:={B}_t\setminus B_1$. The normal derivatives on $\partial B_1$  and $\partial B_t$   are given by
	\begin{equation}\label{eq:annnormder} \frac{\de\tilde f}{\de\nu}(e^{i\theta})=(-\Delta)^{1/2}a+{\mathcal{R}_t}[a,b], \quad \frac{\de\tilde f}{\de\nu}(te^{i\theta})=t^{-1}(-\Delta)^{1/2}b+t^{-1}R_t[b,a], \end{equation}
	where 
	  $R_t:\mathcal{D}'({\mathcal{S}}^1,\R^m)^2\to C^\infty({\mathcal{S}}^1,\R^m)$ is a continuous linear operator defined by 
	 \begin{equation}\label{opR} {\mathcal{R}_t}[a,b](e^{i\theta}):=-\frac{b_0-a_0}{\log t}+\sum_{n>0}\frac{2n}{t^{2n}-1}(a_n-t^n b_n)e^{in\theta}+\sum_{n<0}\frac{2n}{t^{2n}-1}(t^{2n}a_n-t^n b_n)e^{in\theta}. \end{equation}
\end{lemmaen}

\begin{proof}
	Let $\alpha(e^{i\theta}):=\frac{\de\tilde f}{\de\nu}(e^{i\theta})$ and $\beta(e^{i\theta}):=\frac{\de\tilde f}{\de\nu}(te^{i\theta})$. Given any $h\in C^\infty(\de A_t,\env)$, let $c(e^{i\theta}):=h(e^{i\theta})$ and $d(e^{i\theta}):=h(te^{i\theta})$. 
	Since $\tilde f$ is harmonic we get
	\begin{align*} &2\pi\sum_n\alpha_n\cdot\obar{c_n}+2\pi t\sum_n\beta_n\cdot\obar{d_n}=\int_{\de A_t}\frac{\de\tilde f}{\de\nu}h=\int_{A_t}\nabla\tilde f\cdot\nabla\tilde h\\
	&=\int_{{\mathcal{S}}^1}(-\Delta)^{1/4}a(-\Delta)^{1/4}c+\int_{{\mathcal{S}}^1}(-\Delta)^{1/4}b(-\Delta)^{1/4}d
	+{\mathcal{B}_t}((a,b),(c,d)). \end{align*}
	From this equation we easily get \eqref{eq:annnormder}, with ${\mathcal{R}_t}[a,b]$ given by $\eqref{opR}$.
	We observe that the formula \eqref{opR} can also be obtained directly from \eqref{eq:anngrad}.
	The continuity of ${\mathcal{R}_t}$ is proved by arguing as in the  proof of Lemma \ref{bt}.
\end{proof}

\begin{rmk}
	The symmetry ${\mathcal{B}_t}((a,b),(c,d))={\mathcal{B}_t}((b,a),(d,c))$, as well as the fact that the formulas for $t\frac{\de\tilde f}{\de\nu}(te^{i\theta})$ and $\frac{\de\tilde f}{\de\nu}(e^{i\theta})$ can be obtained from each other by exchanging $a$ and $b$, are not surprising in view of the existence of the conformal map
	\[ A_t\to A_t,\quad z\mapsto\frac{tz}{\abs{z}^2}, \]
	which exchanges the two boundary components.
\end{rmk}

\subsection{General compact surfaces with boundary}

The boundary $\de S$ is the disjoint union of finitely many circles diffeomorphic to ${\mathcal{S}}^1$:
\[ \de S=\bigsqcup_{j=1}^k C^{(j)}. \]
We can find, for each $j$, a smooth map
\[ \phi_j:[0,1]\times {\mathcal{S}}^1\to S \]
with the following properties:
\begin{itemize}
	\item $\phi_j$ is a diffeomorphism onto its image;
	\item $\phi_j(\set{0}\times {\mathcal{S}}^1)=C^{(j)}$;
	\item $\phi_j([0,1]\times {\mathcal{S}}^1)\cap\phi_{j'}([0,1]\times {\mathcal{S}}^1)=\emptyset$ for any $j\neq j'$.
\end{itemize}
Applying Theorem \ref{unifann} to the annulus
\[ A^{(j)}:=\phi_j([0,1]\times {\mathcal{S}}^1), \]
we can find a conformal transformation $\psi_j:A^{(j)}\to A_{t_j}$ (where $A_{t_j}:=\obar{B}_{t_j}\setminus B_1\subset\C$, equipped with the flat metric) such that $\psi_j(C^{(j)})=\de B_1$.
Finally, we call
\[ S':=S\setminus\bigsqcup_{j=1}^k\phi_j([0,1)\times {\mathcal{S}}^1). \]

The picture illustrates our decomposition of $S$.

\vspace{0.5cm}
\begin{center}
	\begin{overpic}[width=\linewidth]{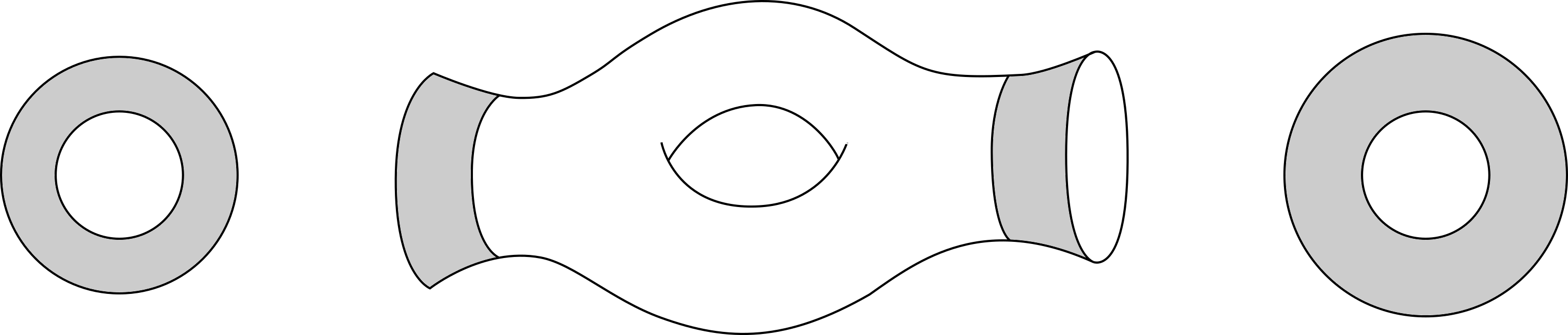}
		\put(6.25,15.25){$A_{t_1}$}
		\put(25.6,9.25){$A^{(1)}$}
		\put(17.25,9.25){\scalebox{1.5}{$\longleftarrow$}}
		\put(19.25,12){{$\psi_1$}}
		\put(21.5,3){$C^{(1)}$}
		\put(89.5,16){$A_{t_2}$}
		\put(74,9.25){\scalebox{1.5}{$\longrightarrow$}}
		\put(75.5,12){{$\psi_2$}}
		\put(63.6,10.5){$A^{(2)}$}
		\put(71.5,4.5){$C^{(2)}$}
	\end{overpic}
\end{center}
\vspace{0.5cm}

We notice that $S'$ is still a smooth surface with boundary
\[ \de S'=\bigsqcup_{j=1}^k\phi_j(\set{1}\times {\mathcal{S}}^1)=\bigsqcup_{j=1}^k\psi_j^{-1}(\de B_{t_j}). \]

\begin{lemmaen}
	For any $f\in H^{1/2}(\de S,\N)$ the $\mz$-energy $E(f)$ admits the decomposition
	\[ E(f,g)=\sum_j \norm{f_j}_{H^{1/2}}^2+{\mathcal{B}}_S((f_j)_{j=1}^k,(f_j)_{j=1}^k), \]
	where $f_j(e^{i\theta}):=f\circ\psi_j^{-1}(e^{i\theta})$ and ${\mathcal{B}}_S:\mathcal{D}'({\mathcal{S}}^1,\env)^k\times\mathcal{D}'({\mathcal{S}}^1,\env)^k\to\R$ is a sequentially continuous symmetric bilinear functional.
\end{lemmaen}

\begin{proof}
	Let $G\in C^\infty((\accentset{\circ}{S}\times S)\setminus\Delta)$ be the Green function for the Dirichlet problem (see e.g. \cite[Theorem~4.17]{Aub98}), satisfying for each $x\in\accentset{\circ}{S}$
	\[ \begin{cases}-\Delta_g G(x,\cdot)=\delta_x & \text{on }S \\ G(x,\cdot)=0 & \text{on }\de S \end{cases} \]
	and let $H\in C^\infty(\accentset{\circ}{S}\times\de S)$ which is defined, for any fixed $x\in\accentset{\circ}{S}$, by the formula $H(x,\cdot):=-\frac{\de}{\de\nu}G(x,\cdot)$. For any $f\in H^{1/2}(\de S,\R^m)$ and any $x\in\interna{S}$, the harmonic extension is given by the formula
	\[ \tilde f(x)=\int_{\de S}H(x,y)f(y)\,\dvg(y). \]
	Now $\set{H(x,\cdot)\mid x\in S'}$ is a compact subset
	of $C^\infty(\de S)$ and in particular is bounded in $C^k(\de S)$ for all $k\ge 0$. The same holds for the derivatives of any order in $x$. Therefore the map
	\[ \mathcal{D}'(\de S,\R^m)\to C^\infty(S',\R^m),\quad f\mapsto\restr{\tilde f}{S'} \]
	given by the above formula is sequentially continuous. In particular, $(f,h)\mapsto\int_{S'}\big\langle\nabla\tilde f,\nabla\tilde h\big\rangle\,\dvg$ defines a sequentially continuous symmetric bilinear operator on $\mathcal{D}'(\de S,\env)\times\mathcal{D}'(\de S,\env)$.

	Moreover, for any $j\in\set{1,\dots,k}$, let
	\[ \kappa_j(e^{i\theta}):=\tilde f\circ\psi_j^{-1}(t_je^{i\theta})\in H^{1/2}({\mathcal{S}}^1,\R^m). \]
	By conformal invariance we have $\Delta(\tilde f\circ\psi_j^{-1})=0$ on $A_{t_j}$ and
	\begin{align*} E^{(j)}\pa{\restr{\tilde f}{\de A^{(j)}},g}=\int_{A_{t_j}}\babs{\nabla(\tilde f\circ\psi_j^{-1})}^2 dx &=\bnorm{(-\Delta)^{1/4}f_j}_{L^2}^2+\bnorm{(-\Delta)^{1/4}\chi_j}_{L^2}^2 \\
	&\quad+{\mathcal{B}}_{t_j}((f_j,\kappa_j),(f_j,\kappa_j)) \end{align*}
	by Lemma \ref{anndec}. We remark that $f\mapsto \kappa_j$ is sequentially continuous as a linear map $\mathcal{D}'(\de S,\R^m)\to C^\infty({\mathcal{S}}^1,\R^m)$.
	Finally, we can write
	\[ E(f,g)=\sum_{j=1}^k\int_{A^{(j)}}\babs{\nabla\tilde f}^2\,\dvg+\int_{S'}\babs{\nabla\tilde f}^2\,\dvg
	=\sum_{j=1}^k\bnorm{(-\Delta)^{1/4}f_j}_{L^2}^2+{\mathcal{B}}_S((f_j),(f_j)), \]
	where for any $f,h\in H^{1/2}(\de S,\env)$ we let
	\begin{align*} {\mathcal{B}}_S((f_j),(h_j))&:=\sum_{j=1}^k\int_{{\mathcal{S}}^1}(-\Delta)^{1/4}\kappa_j(-\Delta)^{1/4}\xi_j+\sum_{j=1}^k B_{t_j}((f_j,\kappa_j),(h_j,\xi_j)) \\
	&\quad+\int_{S'}\big\langle\nabla\tilde f;\nabla\tilde h\big\rangle\,\dvg, \end{align*}
	where $h_j(e^{i\theta}):=h\circ\psi_j^{-1}(e^{i\theta})$ and $\xi_j(e^{i\theta}):=\tilde h\circ\psi_j^{-1}(t_je^{i\theta})$.
\end{proof}

\begin{lemmaen}\label{decnd}
	For any $\ell=1,\dots,k$, the normal derivative on $C^{(\ell)}$ is given by
	\[ \frac{\de\tilde f}{\de\nu}=e^{\lambda_\ell}((-\Delta)^{1/2}f_j)\circ\psi_\ell+e^{\lambda_\ell}R_\ell((f_j)_{j=1}^k)\circ\psi_\ell, \]
	where $e^{2\lambda_\ell}$ is defined by $g=e^{2\lambda_\ell}\psi_\ell^*(g_{\R^2})$ and $R_\ell:\mathcal{D}'({\mathcal{S}}^1,\R^m)^k\to C^\infty({\mathcal{S}}^1,\R^m)$ is a sequentially continuous linear operator.
\end{lemmaen}

\begin{proof}
	Indeed, for any $\varphi\in C^\infty(S,\R^m)$ supported in $\phi_\ell([0,1)\times {\mathcal{S}}^1)$,
	\begin{align*} &\int_{C^{(\ell)}}e^{\lambda_\ell}\varphi\cdot((-\Delta)^{1/2}f_\ell)\circ\psi_\ell\,\dvg
	=\int_{{\mathcal{S}}^1}\varphi\circ\psi_\ell^{-1}\cdot(-\Delta)^{1/2}f_\ell \\
	&=\int_{\de A_{t_j}}\varphi\circ\psi_\ell^{-1}\cdot\frac{\de(\tilde f\circ\psi_\ell^{-1})}{\de\nu}
	-\int_{{\mathcal{S}}^1}\varphi\circ\psi_\ell^{-1}\cdot R_{t_\ell}[f_\ell,g_\ell] \\
	&=\int_{A_{t_\ell}}\big\langle\nabla(\varphi\circ\psi_\ell^{-1});\nabla(\tilde f\circ\psi_\ell^{-1})\big\rangle
	-\int_{C^{(\ell)}}e^{\lambda_\ell}\varphi\cdot R_{t_\ell}[f_\ell,g_\ell]\circ\psi_\ell, \end{align*}
	where the operator $R_{t_\ell}$ is provided by \eqref{eq:annnormder}.
	But, by conformality of $\psi_\ell$,
	\[ \int_{A_{t_\ell}}\big\langle\nabla(\varphi\circ\psi_\ell^{-1});\nabla(\tilde f\circ\psi_\ell^{-1})\big\rangle
	=\int_S\big\langle\nabla\varphi;\nabla\tilde f\big\rangle\,\dvg=\int_{\de S}\varphi\cdot\frac{\de\tilde f}{\de\nu}\,\dvg \]
	and thus we can let $R_\ell((f_j)_{j=1}^k):=R_{t_\ell}[f_\ell,\kappa_\ell]$.
\end{proof}

\section{Conformality of the harmonic extension}\label{confsec}

This section is devoted to show that, if the energy of the harmonic extension $\tilde{u}$ is also critical with respect to variations of the conformal class, then $\tilde{u}$ is conformal.
We will use the \teich space $\tch(S)$ of $S$ to describe such variations. Throughout the section we will assume that $S$ is orientable (actually this hypothesis can be dropped: one can repeat the same theory on the two-sheeted oriented cover $\tilde S$, restricting to equivariant metrics and variations).

We will start from the easier case of the annulus, which can be treated in an elementary fashion (due to the simple explicit form of its \teich space).

\begin{rmk}
	In the disk case, i.e. $S=\mathbb{D}$, conformality holds automatically for $\mz$-harmonic maps (and indeed in this case $\tch(S)$ is trivial).
\end{rmk}

Recall that the disk and the annulus have Euler characteristic $1$ and $0$, respectively. If the surface $S$ has a different topology, then its Euler characteristic is
\[ \chi(S)=2-2g-k<0 \]
(with an abuse of notation, we denote by $g$ also the genus of $S$, while $k\ge 1$ is the number of boundary components). Thus $S$ is intrinsically \emph{hyperbolic}, namely by Gauss--Bonnet theorem any constant curvature metric such that $\de S$ is totally geodesic must have negative curvature. In this case $\tch(S)$ does not possess an immediate presentation as for the annulus, although it is well-known that it is diffeomorphic to $\R^{6g+3k-6}$ (and can be parametrized by means of the so-called Fenchel--Nielsen coordinates). The precise definition of $\tch(S)$ is given below.

\subsection{The annular case}

If $S$ is diffeomorphic to an annulus, then up to a conformal diffeomorphism we can assume that $(S,g)=(A_t,g_{\R^2})$ for some $t>0$, thanks to Theorem \ref{unifann}. A variation of the conformal class (or, more precisely, the conformal class up to diffeomorphisms isotopic to the identity) corresponds to a variation of the parameter $t$.
 
For any $a,b\in H^{1/2}({\mathcal{S}}^1,\R^m)$, we define  $u[t]\in H^{1/2}(\de A_t,\R^m)$  by $u[t](e^{i\theta}):=a(e^{i\theta})$, $u[t](te^{i\theta}):=b(e^{i\theta})$. We will denote by
$\tu[t]$ the harmonic extension of $u[t].$
\begin{lemmaen}\label{varannulus}
	For any  $a,b\in H^1({\mathcal{S}}^1,\R^m)$ we have
	\[ \frac{d}{dt}E_t(u[t])=\int_{\de B_t}\pa{\frac{1}{t^2}\abs{\frac{\de\tu}{\de\theta}}^2-\abs{\frac{\de\tu}{\de r}}^2}. \]
\end{lemmaen}

\begin{proof}
	We can assume $a,b\in C^{\infty}(\mathcal S^1,\env)$ (by a density argument, using the fact that $E_t(u[t])$ depends smoothly on $(t,a,b)\in(1,\infty)\times H^{1/2}({\mathcal{S}}^1)\times H^{1/2}({\mathcal{S}}^1)$, as can be seen from the explicit formula \eqref{eq:annenergyint}). So $\tilde u[t](z)$ defines a smooth function on the set
	\[ \set{(t,z)\in (1,+\infty)\times\C:1\le\abs{z}\le t} \]
	(here $\tilde u[t]\in H^1(A_t,\R^m)$ is the harmonic extension of $u[t]$).
	By the divergence theorem we have
	\begin{align*} \frac{d}{dt}\pa{\int_{A_t}\abs{\nabla\tu[t]}^2}=&\int_{\de B_t}\abs{\nabla\tu[t]}^2+2\int_{A_t}\big\langle\nabla\tu[t];\nabla\pa{\left.\frac{d\tu[s]}{ds}\right|_{s=t}}\big\rangle \\
	=&\int_{\de B_t}\pa{\abs{\frac{\de\tu[t]}{\de r}}^2+\frac{1}{t^2}\abs{\frac{\de\tu[t]}{\de\theta}}^2}+2\int_{\de B_t}\frac{\de\tu[t]}{\de r}\cdot\pa{\left.\frac{d\tu[s]}{ds}\right|_{s=t}} \end{align*}
	(as $\frac{d\tu}{ds}=0$ on $\de B_1$). Differentiating the identity $\tu[s](se^{i\theta})=b(e^{i\theta})$ in $s$ we get
	\[ \left.\pa{\frac{d\tu[s]}{ds}}(se^{i\theta})\right|_{s=t}=-\frac{\de\tu[t]}{\de r}(t e^{i\theta}). \]
	Hence, combining these identities,
	\[ \frac{d}{dt}E_t(u[t])=\int_{\de B_t}\pa{\abs{\frac{\de\tu[t]}{\de r}}^2+\frac{1}{t^2}\abs{\frac{\de\tu[t]}{\de\theta}}^2}
	-2\int_{\de B_t}\abs{\frac{\de\tu[t]}{\de r}}^2
	. \qedhere \]
\end{proof}

\begin{proof}[Proof of Theorem \ref{confannIntr}]We introduce the Hopf differential
	\[ h(z):=H(z)\,dz\otimes dz,\quad H(z):=\frac{\de\tu}{\de z}\cdot\frac{\de\tu}{\de z}=\frac{1}{4}\pa{\abs{\frac{\de\tu}{\de x}}^2-\abs{\frac{\de\tu}{\de y}}^2-2i\frac{\de\tu}{\de x}\cdot\frac{\de\tu}{\de y}}. \]
	A well-known straightforward computation shows that $h$ is a holomorphic quadratic differential, i.e. $H$ is holomorphic, vanishing identically if and only if $\tilde u$ is (weakly) conformal. From Theorem \ref{regul} it follows   in particular that  $u\in C^1(\partial A_t)$.  
	Since $0 = P^T(u)\,\partial_{r}\tu$ and $0=P^N(u)\,\partial_{\theta}u=P^N(u)\,\partial_{\theta}\tu$ on $\partial A_t$, we have that
	\[ \frac{\de\tu}{\de r}\cdot\frac{\de\tu}{\de\theta}=0\text{ on }\de A_\obt. \]
	By the maximum principle we deduce that,
	for  any $z=re^{i\theta}\in\iA_\obt$,
	\[ -2\Im((re^{i\theta})^2 H(re^{i\theta}))=r\frac{\de\tu}{\de r}(re^{i\theta})\cdot\frac{\de\tu}{\de\theta}(re^{i\theta})=0. \]
	
	Thus the harmonic function $\Im(z^2 H(z))$ vanishes identically. Since $z^2 H(z)$ is holomorphic,
	it must coincide with a real constant $c$. By Lemma \ref{varannulus}, $c=0$ precisely when $\left.\frac{d}{dt}E_t(u[t])\right|_{t=\obar{t}}=0$, since
	\[ -4c=-4\Re((re^{i\theta})^2 H(re^{i\theta}))=\abs{\frac{\de\tu}{\de\theta}(re^{i\theta})}^2-r^2\abs{\frac{\de\tu}{\de r}(re^{i\theta})}^2. \qedhere \]
\end{proof}

\begin{rmk}\label{diskconf}
	In the disk case we get $z^2 H(z)=c$ for some real $c$, hence $c=0$ (being $H$ bounded near the origin) and $H(z)=0$.
\end{rmk}

\subsection{The hyperbolic case}

%

Assume now that $\chi(S)<0$ (i.e. $S$ is not a disk nor an annulus) and $\N$ is $C^\infty$-smooth.
Let $\mathcal{M}(S)$ be the space of all Riemannian metrics on $S$ and $\mathcal{P}(S)$ the space of all smooth positive functions $S\to\R$. $\mathcal{M}(S)$ is an open subset of the Fr\'echet space $\Gamma(S^2 S)$ (smooth symmetric covariant $2$-tensors on $S$).
The quotient
\[ \mathcal{C}(S):=\mathcal{M}(S)/\mathcal{P}(S) \]
is the set of conformal classes on $S$. Moreover, let $\mathcal{M}_{-1}(S)\subseteq\mathcal{M}(S)$
be the subset of metrics having constant curvature $-1$ and making $\de S$ totally geodesic.
Every equivalence class $[g]\in\mathcal{C}(S)$ has exactly one representative $e^{2\lambda}g\in\mathcal{M}_{-1}(S)$, $\lambda\in C^\infty(S)$ being a solution of Liouville's equation
\[ \begin{cases}\Delta\lambda=K+e^{2\lambda} & \text{on }S \\ \frac{\de\lambda}{\de\nu}=\kappa & \text{on }\de S, \end{cases} \]
where $K$ is the Gaussian curvature of $g$ and $\kappa$ is the geodesic curvature of the boundary (i.e. $\kappa=\ang{\nabla_{\dot\gamma}\dot\gamma,\nu}$ if $\de S$ is locally parametrized by a unit-speed curve $\gamma$).
The map
\[ \upsilon:\mathcal{M}(S)\to\mathcal{M}_{-1}(S), \quad \upsilon(g):=e^{2\lambda}g \]
is $C^\infty$-smooth (as a map from $\mathcal{M}(S)$ into itself).

In order to have a finite-dimensional space, we quotient $\mathcal{C}(S)$ by the (right) action of the group $\mathcal{D}_0(S)$ of diffeomorphisms isotopic to the identity.
The set
\[ \tch(S):=\mathcal{C}(S)/\mathcal{D}_0(S) \]
is the \teich space of $S$. It can be given a canonical structure of $(6g+3k-6)$-dimensional differentiable manifold. The resulting map $\pi:\mathcal{M}(S)\to\tch(S)$ is smooth and admits locally a smooth section taking values into $\mathcal{M}_{-1}(S)$.

For the proofs of these facts, we refer the reader to \cite{FT84}, where the \teich theory for closed surfaces is developed. See also \cite[Section~4.3]{DHT92}, which illustrates the necessary modifications for surfaces with boundary (using the convenient device of the Schottky double).

\begin{thm}\label{hypconf}
	Let $(S,g)$ be a Riemannian surface with $\de S\neq\emptyset$, $\chi(S)<0$ and let $\phi:U\to\mathcal{M}(S)$ be a local smooth section of $\pi$ through $g$ (i.e. $\pi(g)\in U$ and $\phi(\pi(g))=g$). If $u\in H^{1/2}(\de S,\mathcal{N})$ is $\mz$-harmonic with respect to $g$, then $\tilde{u}:(S,g)\to\R^m$ is conformal if and only if $\pi(g)$ is a critical point for the map
	\[ p\mapsto E(u,\phi(p)). \]
\end{thm}

We remark that the harmonic extension $\tilde u_p\in H^1(S)$ of $u\in H^{1/2}(\de S)$ with respect to $\phi(p)$ depends on the couple $(u,p)\in H^{1/2}(\de S,\R^m)\times U$ in a smooth fashion:
this follows from the inverse function theorem applied to the map
\[ H^1(S)\times U\to H^{-1}(S)\times H^{1/2}(\de S)\times U,\quad (v,p)\mapsto (-\Delta_{\phi(p)}v,v\big|_{\de S},p). \]
In particular, the function $(u,p)\mapsto E(u,\phi(p))$ is smooth as well.

\begin{proof}
	$(\Rightarrow)$ Replacing $g$ with $\upsilon(g)$ and $\phi$ with $\upsilon\circ\phi$, we can assume that $g\in\mathcal{M}_{-1}(S)$ and $\phi(U)\subseteq\mathcal{M}_{-1}(S)$: indeed,
	thanks to the conformal invariance of the Dirichlet energy, $E(v,g')=E(v,\upsilon(g'))$ for any $v\in H^{1/2}(\de S,\R^m)$ and any metric $g'$, so $u$ is still $\mz$-harmonic with respect to $\upsilon(g)$ and $\pi(\upsilon(g))=\pi(g)$ is still critical for $p\mapsto E(u,\upsilon\circ\phi(p))$.
	
	The Hopf differential $h$ of the map $\tilde{u}$, defined in any local conformal chart (for $g$) by the formula
	\[ h(z):=H(z)\,dz\otimes dz,\quad H(z):=\frac{\de\tu}{\de z}\cdot\frac{\de\tu}{\de z}=\frac{1}{4}\pa{\abs{\frac{\de\tu}{\de x}}^2-\abs{\frac{\de\tu}{\de y}}^2-2i\frac{\de\tu}{\de x}\cdot\frac{\de\tu}{\de y}}, \]
	is a globally defined holomorphic quadratic differential (i.e. $H$ is holomorphic in any conformal chart), as a consequence of the fact that $\Delta_g\tilde{u}=0$. The conformality of $\tilde{u}$ is equivalent to $h\equiv 0$.
	
	Moreover, $h$ is real at the boundary $\de S$, meaning that in any local conformal chart $(V,z)$ mapping $V\cap\de S$ into the real line $\set{\Re(z)=0}$
	we have
	\[ \frac{\de\tilde{u}}{\de x}\cdot\frac{\de\tilde{u}}{\de y}=0 \]
	on the real line. Indeed, at such points $\frac{\de\tilde{u}}{\de x}\in T_{u(z)}\mathcal{N}$, while $\frac{\de\tilde{u}}{\de y}\perp T_{u(z)}\mathcal{N}$ by $\mz$-harmonicity (observe that by the preceding regularity result we have $\tilde{u}\in C^\infty$ up to the boundary).
	
	Let now $v:=d\pi_{g}[\Re(h)]$. Since $g\in\mathcal{M}_{-1}(S)$, the symmetric tensor $d\phi_{\pi(g)}[v]$ can be decomposed as
	\begin{equation}\label{eq:hodgedec} d\phi_{\pi(g)}[v]=\Re(q)+\mathcal{L}_X g, \end{equation}
	where $q$ is a holomorphic quadratic differential which is real at $\de S$, while $\mathcal{L}_X g$ is the Lie derivative of $g$ with respect to a vector field $X$ satisfying $\restr{X}{\de S}\parallel\de S$ (see \cite[Theorem~8.2]{FT84} for the corresponding statement for closed surfaces). The tensor $\mathcal{L}_X g$ belongs to the kernel of $d\pi_g$, as $X$ generates a one-parameter subgroup of $\mathcal{D}_0$. Thus, using $\pi\circ\phi=\text{id}_U$,
	\[ d\pi_g[\Re(h)]=v=d\pi_g[d\phi_{\pi(g)}[v]]=d\pi_g[\Re(q)+\mathcal{L}_X g]=d\pi_g[\Re(q)]. \]
	But $\tch(S)$ is built precisely in such a way that the map $d\pi_g$ restricts to a bijection from the space of such real quadratic differentials to $T_{\pi(g)}\tch(S)$. We deduce that $\Re(h)=\Re(q)$.
	
	We also remark that $dE(u,\cdot)_g[\mathcal{L}_X g]=0$:
	indeed, calling $\Phi_t^X$ the flow generated by $X$, we clearly have
	\[ E(u,(\Phi_t^X)^*g)=E(u\circ\Phi_{-t}^X,g) \]
	and differentiation at $t=0$ gives
	\[ dE(u,\cdot)_g[\mathcal{L}_X g]=-2\int_{\de S}\frac{\de\tilde{u}}{\de\nu}\cdot du[X]\,d\text{vol}_g=0, \]
	by characterization \eqref{eq:hhfchar} of $\mz$-harmonicity.
	From \eqref{eq:hodgedec} we finally deduce that
	\begin{align*} 0=&dE(u,\phi(\cdot))[v]=dE(u,\cdot)[\Re(q)+\mathcal{L}_X g]=dE(u,\cdot)[\Re(h)] \\
	=&-\int_S\Re(h)[\nabla\tilde u;\nabla\tilde u]\,\dvg+\mz\int_S\abs{\nabla\tilde u}_g^2\text{tr}_g(\Re(h))\,\dvg \end{align*}
	(using the fact that the variation of $\tilde{u}$ gives no contribution, thanks to harmonicity). But, as is readily seen in conformal coordinates,
	\[ \text{tr}_g(\Re(h))=0,\quad \Re(h)[\nabla\tilde u,\nabla\tilde u]
	=2\abs{\Re(h)}_g^2. \]
	We infer that $\Re(h)=0$, which implies $h=0$.
	
	$(\Leftarrow)$ Conversely, for any $v\in T_{\pi(g)}U$, we can write
	\[ d\phi_{\pi(g)}[v]=\Re(q)+\mathcal{L}_X g \]
	for suitable $q$ and $X$ as before. We have
	\begin{align*} &dE(u,\phi(\cdot))[v]=dE(u,\cdot)[\Re(q)+\mathcal{L}_X g]=dE(u,\cdot)[\Re(q)] \\
	&=-\int_S\Re(q)[\nabla\tu;\nabla\tu]\,\dvg+\mz\int_S\abs{\nabla\tu}_g^2\text{tr}_g(\Re(q))\,\dvg. \end{align*}
	Again we have $\text{tr}_g(\Re(q))=0$, while the conformality of $\tu$ gives $\frac{\de\tu}{\de z}\cdot\frac{\de\tu}{\de z}=0$ in conformal coordinates, hence $\Re(q)[\nabla\tu;\nabla\tu]=0$ as well.
\end{proof}

\begin{proof}[Proof of Corollary \ref{cvx}]
	In view of the preceding results, it suffices to show that, for a nontrivial $\mz$-harmonic map $u$ with conformal $\tilde u$, we have $\tilde u(\interna{S})\subseteq\Omega$ and $d\tilde u\neq 0$ at the boundary $\de S$. Recall that $\tilde u$ is $C^\infty$-smooth up to the boundary of $S$.
	
	Since $u$ is nontrivial we have $d\tilde u\neq 0$ at some $x'\in\de S$. Combining this with the conformality of $\tilde u$ and the condition $P^T(u)\frac{\de\tilde u}{\de\nu}=0$, we get $\tilde u(x'')\in\Omega$ for at least an $x''\in\interna{S}$.
	
	Fix now any point $p\nin\Omega$. By convexity of $\Omega$, there exists an affine map $F:\env\to\R$ such that $F(p)\le0$ and $F(\Omega)\subseteq(0,\infty)$.
	Since $F\circ u$ takes nonnegative values (as $u$ takes values in $\de\Omega$) and $F\circ\tilde u(x'')>0$, by the strong maximum principle we get $F\circ\tilde u>0$ on $\interna{S}$. Hence, $\tilde u(\interna{S})\subseteq\Omega$.
	
	Finally, if $x\in\de S$ we can let $p:=u(x)$: then Hopf's lemma gives $\frac{\de(F\circ\tilde u)}{\de\nu}(x)<0$. In particular, $d\tilde u$ never vanishes at $\de S$.
\end{proof}

\appendix

\section{Uniformization theorem for annuli with boundary}\label{unifannsec}

\begin{thm}\label{unifann} Let $(A,g)$ be a compact Riemannian two-dimensional manifold with boundary, diffeomorphic to $[0,1]\times {\mathcal{S}}^1$.
Then there exists some $t>1$ such that $(A,g)$ is conformally equivalent to the standard annulus $A_t:=\obar{B}_t\setminus B_1\subset\C$.
\end{thm}

\begin{proof}We fix a diffeomorphism $\phi:[0,1]\times {\mathcal{S}}^1\to A$ and we orient $A$ by declaring that $\phi$ is orientation-preserving.
We call $\gamma_j:{\mathcal{S}}^1\to\de A$ the restrictions $\gamma_j:=\phi(j,\cdot)$, for $j=0,1$, so that $\gamma_1$ preserves the orientation while $\gamma_0$ reverses it.
Let $u\in C^\infty(A)$ be the unique harmonic function which equals $j$ on $\gamma_j({\mathcal{S}}^1)$ (for $j=0,1$). Denoting by $\iA$ the interior of $A$,
we remark that by the maximum principle $0<u<1$ on $\iA$ and by Hopf's lemma $*du[\dot\gamma_j]>0$ for $j=0,1$. Recall that, in local conformal coordinates $(x,y)$,
$*du=-\frac{du}{dy}dx+\frac{du}{dx}dy$.

Let $\kappa:=\int_{\gamma_0}*du>0$. Since $*du$ is closed, $\int_\gamma *du\in\kappa\Z$ for any closed, piecewise smooth curve $\gamma$ taking values in $A$.
Thus, we can define $v\in C^\infty(A,\R/\kappa\Z)$ by the formula $v(p):=\int_\alpha *du$, where $\alpha$ is any piecewise smooth curve joining $\gamma_0(0)$ to $p$.
Now the map
\[ \psi:A\to\C, \quad \psi:=\exp\pa{\frac{2\pi}{\kappa}(u+iv)} \]
is well-defined and smooth.

The metric $g$, together with the orientation, induces a complex structure on $A$.
As $v$ locally lifts to a primitive of $*du$, we have $dv=*du$. Hence, in local conformal coordinates, the map $u+iv:A\to\C/i\kappa\Z$ satisfies the Cauchy--Riemann equations and is thus holomorphic; so $\psi$ is holomorphic as well.
We now prove that $\psi$ is a diffeomorphism onto its image. Since $*du[\dot\gamma_i]>0$, the compact set $F:=\set{p\in A:d\psi(p)=0}$ is contained in $\iA$.
As $\psi$ is holomorphic, $F$ is finite. We have $F':=\psi^{-1}(\psi(F))\subseteq\iA$ (as $\psi(\de A)\cap\psi(\iA)=\emptyset$), so by holomorphicity $F'$ is finite as well.

It is easy to check that $\restr{\psi}{A\setminus F'}:A\setminus F'\to\psi(A)\setminus\psi(F)$ is a covering (indeed, any $z\in\psi(A)\setminus\psi(F)$ has finitely many preimages $p_1,\dots,p_k\in A\setminus F'$; we can find open disjoint neighborhoods $U_j\subseteq A\setminus F'$ of $p_j$ which are all mapped diffeomorphically onto some neighborhood $V$ of $z$; up to replacing $V$ with $V':=V\setminus\psi(A\setminus\sqcup_j U_j)$ and shrinking each $U_j$ accordingly, $V$ is evenly covered by $\sqcup_j U_j$).
But $\psi$ is injective on $\de A\subseteq A\setminus F'$, so $\restr{\psi}{A\setminus F'}$ is injective and hence a diffeomorphism onto its image.

As $\psi$ is holomorphic, $\psi$ cannot be injective in any punctured neighborhood of any point in $F$. It follows that $F=\emptyset$, thus also $F'=\emptyset$. Finally, calling $t:=\frac{2\pi}{\kappa}$, we have $\psi(A)\subseteq A_t$ and $\psi(\de A)=\de A_t$. As $\psi(\iA)=\psi(A)\cap\iA_t$, $\psi(\iA)$ is both open and closed in $\iA_t$, so by connectedness it follows that $\psi:A\to A_t$ is surjective. The map $\psi$ provides the desired conformal equivalence.
\end{proof}

\section{Functional spaces}\label{functapp}

In this section we recall the definition of the functional spaces used in the paper, as well as the main operations on them and some key facts concerning the so-called Littlewood--Paley  dyadic  decomposition.

We denote respectively by ${\mathcal{S}}(\R)$ and  ${\mathcal{S}}^{\prime}(\R)$ the spaces of (real or complex) Schwartz functions and tempered distributions.
All the functional spaces used in this paper should be understood as subspaces of $\mathcal{S}'(\R)$.
Given a function $\varphi\in\mathcal{S}(\R)$, we denote either by
$\hat\varphi$ or by ${\mathcal{F}}\varphi$ the Fourier transform of $\varphi$, i.e.
\[ \hat\varphi(\xi)={\mathcal{F}}\varphi(\xi)=\int_{\R}v(x)e^{-2\pi i\xi x}\,dx, \]
while if $v\in\mathcal{S}'(\R)$ we define $\hat v=\mathcal{F}v\in\mathcal{S}'(\R)$ by $\ang{\hat v,\varphi}:=\ang{v,\hat\varphi}$.

We recall the definition of the inhomogeneous fractional Sobolev (Bessel potential) spaces:
for a real $s$ and $1<p<\infty$ we let
\[ H^{s,p}(\R):=\set{v\in\mathcal{S}'(\R):\norm{v}_{H^{s,p}}:=\bnorm{{\mathcal{F}}^{-1}[(1+4\pi^2|\xi|^2)^{s/2}{\mathcal{F}}v]}_{L^p}<\infty}. \]
Observe that $H^{s,p}(\R)$ is stable under multiplication by Schwartz functions, i.e. $\psi v\in H^{s,p}(\R)$ if $\psi\in\mathcal{S}(\R)$ and $v\in H^{s,p}(\R)$: assuming $v\in\mathcal{S}(\R)$, $w:={\mathcal{F}}^{-1}[(1+|\xi|^2)^{s/2}{\mathcal{F}}v]\in L^p(\R)$ satisfies
\[ {\mathcal{F}}^{-1}[(1+4\pi^2|\xi|^2)^{s/2}{\mathcal{F}}(\psi v)]=\int\mathcal{F}^{-1}[m_\eta \mathcal{F}(e^{2\pi i\eta x}w)]\,d\eta,\quad m_\eta(\xi):=\pa{\frac{1+4\pi^2\abs{\xi}^2}{1+4\pi^2\abs{\xi-\eta}^2}}^{s/2}\hat\psi(\eta) \] and we get $\norm{\psi v}_{H^{s,p}}\ls\norm{w}_{L^p}=\norm{v}_{H^{s,p}}$ from \cite[Theorem~6.2.7]{Gra14C} and the decay of $\hat\psi$.


We also recall the definition of the homogeneous fractional Sobolev spaces used in the paper, namely $\hmz(\R)$ and $\dot H^{-1/2}(\R)$:
\[ \dot H^{1/2}(\R):=\bigg\{v\in L^2_{loc}(\R):\norm{v}_\hmz^2:=\iint\frac{\abs{v(x)-v(y)}^2}{\abs{x-y}^2}\,dx\,dy<\infty\bigg\}, \]
\[ \dot H^{-1/2}(\R):=\set{v\in\mathcal{S}'(R):\hat v\in L^2_{loc}(\R)\text{ and }\int\abs{\xi}^{-1}|\hat v(\xi)|^2<\infty}. \]
We remark that $\dot H^{1/2}(\R)$ is naturally a subspace of $\mathcal{S}'(\R)$, although $\norm{\cdot}_\hmz$ is only a seminorm (which vanishes on constant functions). Given $v\in\hmz(\R)$, we always have $\hat v\in L^2_{loc}(\R\setminus\set{0})$ and
\begin{equation}\label{hmzfourier} \int_{\R\setminus\set{0}}\abs{\xi}|\widehat v(\xi)|^2\,d\xi=c\norm{v}_\hmz^2 \end{equation}
for some constant $c>0$ (see e.g. the proof of \cite[Proposition~3.4]{DNPV12}).

\begin{lemmaen}\label{hmzannest}
	Given $v\in\dot H^{1/2}(\R)$, for any $j\ge 0$ it holds
	\[ \norm{v}_{L^2(B(0,2^{j}))}\ls\abs{(v)_{B(0,1)}}+(j+1)2^{j/2}\norm{v}_{\dot H^{1/2}}. \]
\end{lemmaen}

\begin{proof}
	We notice that, for $j\ge 0$,
	\[ \norm{v-(v)_{B(0,2^j)}}_{L^2(B(0,2^j))}^2\ls 2^{-j}\iint_{B(0,2^j)^2}\abs{v(x)-v(y)}^2\ls 2^j\iint_{B(0,2^j)^2}\frac{\abs{v(x)-v(y)}^2}{\abs{x-y}^2}, \]
	so $\norm{v-(v)_{B(0,2^j)}}_{L^2(B(0,2^j))}\ls 2^{j/2}\norm{v}_\hmz$. Similarly,
	\[ \abs{(v)_{B(0,2^{j-1})}-(v)_{B(0,2^j)}}\ls 2^{-j}\int_{B(0,2^j)}\abs{v-(v)_{B(0,2^j)}}\ls\norm{v}_\hmz \]
	for $j\ge 1$. The desired inequality follows from these estimates and
	\[ \norm{v}_{L^2(B(0,2^j))}\le\norm{v-(v)_{B(0,2^j)}}_{L^2(B(0,2^j))}+2^{j/2}\sum_{\ell=1}^j\abs{(v)_{B(0,2^{\ell-1})}-(v)_{B(0,2^\ell)}}+\abs{(v)_{B(0,1)}}. \qedhere \]
\end{proof}

\begin{lemmaen}\label{hmzapprox}
	Given $v\in\hmz(\R)$, there exists a sequence $v_k\in\mathcal{S}(\R)$, with $\widehat{v}_k\in C^\infty_c(\R\setminus\set{0})$, and a sequence $c_k\in\R$ such that
	\[ \norm{v-(v_k+c_k)}_\hmz=\norm{v-v_k}_{\hmz}\to 0,\qquad v_k+c_k\weakstarto v\quad\text{in }\mathcal{S}'(\R), \]
	\[ \norm{v-(v_k+c_k)}_{L^2(B(0,2^j))}\ls (j+1)\norm{v-(v_k+c_k)}_\hmz. \]
\end{lemmaen}

\begin{proof}
	Fix $\chi\in C^\infty_c(\R)$ with $\uno_{B(0,1/2)}\le\chi\le\uno_{B(0,1)}$. As observed above, the function
	\[ w_k:=(\chi(2^{-k}\cdot)-\chi(2^k\cdot))\widehat{v} \]
	belongs to $L^2(\R)$ and
	has $\int\abs{\xi}\abs{w_k(\xi)}^2\,d\xi<\infty$. We can find $v_k\in\mathcal{S}(\R)$ with $\widehat{v}_k\in C^\infty_c(\R\setminus\set{0})$ and $\int\abs{\xi}\abs{w_k-\widehat{v}_k}^2(\xi)\,d\xi\le 2^{-k}$. Since
	$\int_{\R\setminus\set{0}}\abs{\xi}\abs{\widehat{v}-w_k}^2(\xi)\,d\xi\to 0$, we get
	\[ \norm{v-v_k}_\hmz^2\ls\int_{\R\setminus\set{0}}\abs{\xi}\abs{\widehat{v}-\widehat{v}_k}^2(\xi)\,d\xi\to 0. \]
	We now choose $c_k$ in such a way that $(v_k+c_k)_{B(0,1)}=(v)_{B(0,1)}$. The last part of the thesis follows from Lemma \ref{hmzannest} and the convergence $v_k+c_k\weakstarto v$ is an immediate consequence.
\end{proof}

\begin{rmk}\label{linftyapprox}
	If $v$ lies also in $L^\infty(\R)$, we can also ensure that $\norm{v_k}_{L^\infty},\abs{c_k}\ls\norm{v}_{L^\infty}$. Indeed, $\mathcal{F}^{-1}(\chi(2^{-k}\cdot)-\chi(2^k\cdot))$ is bounded in $L^1(\R)$, so $\norm{\widecheck w_k}_{L^\infty}\ls\norm{v}_{L^\infty}$; moreover, $v_k$ can be chosen arbitrarily close to $\widecheck w_k$ in $L^\infty(\R)$.
\end{rmk}

We also define the Hardy space $\mathcal{H}^1(\R)$ as
\[ \mathcal{H}^1(\R):=\set{v\in L^1(\R):\sup_{t>0}\abs{\varphi_t*v}(x)\in L^1(\R)}, \]
where $\varphi\in\mathcal{S}(\R)$ is an arbitrary function such that $\int\varphi\neq 0$ and $\varphi_t(y):=t^{-1}\varphi(t^{-1}y)$. This definition does not depend on the choice of $\varphi$: for this and many useful characterizations of $\mathcal{H}^1(\R)$, we refer the reader to \cite[Chapter~2]{Gra14M} and \cite[Chapter~III]{Ste93}.

Finally we define the Lorentz spaces $L^{2,1}(\R)$ and $L^{2,\infty}(\R)$:
\[ L^{2,1}(\R):=\set{v\in L^1_{loc}(\R):\int_0^\infty\mathcal{L}^1(\set{\abs{f}>t})^{1/2}\,dt<\infty}, \]
\[ L^{2,\infty}(\R):=\set{v\in L^1_{loc}:\sup_{t>0}t\mathcal{L}^1(\set{\abs{f}>t})^{1/2}<\infty}. \]
These are Banach spaces with the norms
\[ \norm{v}_{L^{2,1}}:=\int_0^\infty t^{-1/2}\pa{\sup_{t\le\mathcal{L}^1(E)<\infty}\media_E\abs{v}}\,dt,\quad\norm{v}_{L^{2,\infty}}:=\sup_{0<\mathcal{L}^1(E)<\infty}\mathcal{L}^1(E)^{-1/2}\int_E\abs{v} \]
and $L^{2,\infty}(\R)$ is the dual of $L^{2,1}(\R)$: see e.g. \cite[Section~1.4]{Gra14C}.

\subsection{Products, fractional Laplacian and Hilbert--Riesz transform}

We fix a nonnegative bump function $\rho\in C^\infty_c(\R)$ with $\int\rho=1$. Given $v,w\in\mathcal{S}'(\R)$, we define their product
\[ vw:=\lim_{\epsilon\to 0}(\rho_\epsilon*v)(\rho_\epsilon*w) \]
as a limit in $\mathcal{S}'(\R)$, provided that it exists. Notice that $(\rho_\epsilon*v)(\rho_\epsilon*w)\in C^\infty\cap\mathcal{S}'(\R)$. In general, this limit could fail to exist or could depend on $\rho$. In all the instances appearing in this paper, we are implicitly claiming that the product is defined and is independent of $\rho$.

From the definition of $\hmz(\R)$ it easily follows that $\hmzb(\R)$ is an algebra, i.e. it is closed under the product: more precisely,
\[ \norm{vw}_\hmz\ls\norm{v}_\hmz\norm{w}_{L^\infty}+\norm{v}_{L^\infty}\norm{w}_\hmz,\quad\norm{vw}_{L^\infty}\le\norm{v}_{L^\infty}\norm{w}_{L^\infty} \]
whenever $v,w\in\hmzb(\R)$. Using this and the obvious inclusion $\mathcal{S}(\R)\subseteq\hmzb(\R)$, as well as \eqref{hmzfourier}, one checks that the product $vw$ can always be formed when $v\in\dot H^{-1/2}(\R)$ and $w\in\hmzb(\R)$.

Moreover, for any real $s$, we define the fractional Laplacian $(-\Delta)^{s/2}$ as
\[ (-\Delta)^{s/2}v:=\lim_{\epsilon\to 0}\mathcal{F}^{-1}[(\epsilon^2+4\pi^2\abs{\xi}^2)^{s/2}\mathcal{F}v], \]
provided that the limit exist in $\mathcal{S}'(\R)$; in other words, we approximate the fractional Laplacian by means of Bessel potentials. We recall some properties of the fractional Laplacian for the values of $s$ mostly used in the paper, namely $s=\pm\frac{1}{4}$.

Clearly, $\ql$ maps $L^2(\R)$ isomorphically onto $\hmmz(\R)$, with inverse $\mql$. The following statement is less obvious.

\begin{lemmaen}\label{hmztol2}
	If $v\in\hmz(\R)$, then $\ql v$ exists, lies in $L^2(\R)$ and is given by
	\[ \ql v=\mathcal{F}^{-1}\pa{(2\pi\abs{\xi})^{1/2}\widehat{v}}, \]
	where $(2\pi\abs{\xi})^{1/2}\widehat{v}$ means the $L^2$ function agreeing with the distribution on $\R\setminus\set{0}$.
\end{lemmaen}

\begin{proof}
	Call $w$ the $L^2$ function $(2\pi\abs{\xi})^{1/2}\widehat{v}$.
	We observe that $(\epsilon^2+4\pi^2\xi^2)^{1/4}=\epsilon^{1/2}+\frac{\xi^2}{2}\int_0^ 1 4\pi^2t(\epsilon^2+4\pi^2t^2\xi^2)^{-3/4}\,dt$, with the second term vanishing for $\xi=0$.
	Using Lemma \ref{hmzapprox} and \eqref{hmzfourier} we get
	\begin{align*} (\epsilon^2+4\pi^2\xi^2)^{1/4}\widehat{v}&=\epsilon^{1/2}\widehat{v}+\lim_{k\to\infty}\pa{\frac{\xi^2}{2}\int_0^1 4\pi^2t(\epsilon^2+4\pi^2t^2\xi^2)^{-3/4}\,dt}(\widehat{v}_k+\widehat{c}_k) \\
	&=\epsilon^{1/2}\widehat{v}+\pa{\int_0^1 \frac{t}{2}(2\pi\abs{\xi})^{3/2}(\epsilon^2+4\pi^2t^2\xi^2)^{-3/4}\,dt}w. \end{align*}
	Finally, $\epsilon^{1/2}\widehat{v}\weakstarto 0$ in $\mathcal{S}'(\R)$ and the nonnegative integral converges to $\uno_{\R\setminus\set{0}}$ from below.
\end{proof}
A similar proof shows that $\ml v=\mathcal{F}^{-1}(2\pi\abs{\xi}\hat v)$, so $\ml v=\ql\ql v$.

One has also the following integral representation for $\ql v$.

\begin{lemmaen}\label{qlintrep}
	For all $v\in\hmz(\R)$ and some constant $c>0$ (independent of $v$) we have
	\[ \ql v(x)=c\lim_{\epsilon\to 0}\int_{\R\setminus B(x,\epsilon)}\frac{v(x)-v(y)}{\abs{x-y}^{3/2}}\,dy\quad\text{in }L^2(\R). \]
\end{lemmaen}

\begin{proof}
	Let $w_\epsilon(x):=\int_{\R\setminus B(x,\epsilon)}\frac{v(x)-v(y)}{\abs{x-y}^{3/2}}\,dy$ (which lies in $\mathcal{S}'(\R)$ by Lemma \ref{hmzannest}) and take $\varphi\in\mathcal{S}(\R)$. Fubini's theorem gives
	\begin{align*} \ang{\widehat{w}_\epsilon,\varphi}&=\ang{w_\epsilon,\widehat{\varphi}}=\iint_{\R\times(\R\setminus B(0,\epsilon))}\frac{v(x)-v(x+h)}{\abs{h}^{3/2}}\widehat{\varphi}(x)\,dx\,dh \\
	&=\iint_{\R\times(\R\setminus B(0,\epsilon))}\frac{v(x)(\widehat{\varphi}(x)-\widehat{\varphi}(x-h))}{\abs{h}^{3/2}}\,dx\,dh
	=\int_{\R\setminus B(0,\epsilon)}\frac{\ang{\widehat{v},(1-e^{2\pi ihx})\varphi(x)}}{\abs{h}^{3/2}}\,dh. \end{align*}
	Since $(1-e^{2\pi ihx})\varphi(x)$ vanishes at $0$, Lemmas \ref{hmzapprox} and \ref{hmztol2} show that
	\[ \ang{\widehat{v},(1-e^{2\pi ihx})\varphi(x)}=\lim_{k\to\infty}\int\abs{x}^{1/2}\widehat{v}_k(x)\frac{1-e^{2\pi ihx}}{\abs{x}^{1/2}}\varphi(x)\,dx=\int\mathcal{F}(\ql v)\frac{1-e^{2\pi ihx}}{(2\pi\abs{x})^{1/2}}\varphi(x)\,dx. \]
	We conclude that
	\[ \widehat{w_\epsilon}(x)=\mathcal{F}(\ql v)(x)\int_\epsilon^\infty\frac{2-2\cos(2\pi hx)}{\abs{h}^{3/2}(2\pi\abs{x})^{1/2}}\,dh \]
	and, for $x\neq 0$, the last integral equals $\int_{\epsilon\abs{x}}^\infty\frac{2-2\cos (2\pi t)}{(2\pi)^{1/2}t^{3/2}}\,dt$, which converges to some positive constant from below, as $\epsilon\to 0$.
\end{proof}

As for the formal inverse, the Riesz potential operator $\mql$, notice that $\mathcal{F}^{-1}(\abs{\xi}^{-1/2})=c\abs{x}^{-1/2}$ for some $c\in\R$ (indeed, $\abs{x}^{-1/2}$ is the only $-\mz$-homogeneous tempered distribution up to multiples: see e.g. \cite[Proposition~2.4.8]{Gra14C}).

Since $\abs{x}^{-1/2}\in L^{2,\infty}(\R)$, we get
$\mql(L^1(\R))\subseteq L^{2,\infty}(\R)$ and $\mql(L^{2,1}(\R))\subseteq L^\infty(\R)$.\footnote{For $v\in L^{2,1}(\R)$ the fractional Laplacian $\ql v$ exists and equals $c\abs{x}^{-1/2}*v$: indeed, from \cite[Proposition~1.2.5]{Gra14M} one easily deduces the weak* convergence of $\mathcal{F}^{-1}[(\epsilon^2+4\pi^2\abs{\xi}^2)^{-1/4}]$ to $\mathcal{F}^{-1}(\abs{\xi}^{-1/2})$ in $L^{2,\infty}(\R)$.} Also, $\mql$ maps $\hardy(\R)$ into $L^{2,1}(\R)$: this is a straightforward consequence of the atomic decomposition property of $\hardy(\R)$ (see \cite[Section~III.2]{Ste93}).

Finally, we define the Hilbert--Riesz transform of $v\in\mathcal{S}'(\R)$ as
\[ \rz v:=\lim_{\epsilon\to 0}\mathcal{F}^{-1}\bra{-i\frac{\xi}{(\epsilon^2+\abs{\xi}^2)^{1/2}}\hat v}, \]
whenever the limit exists.
A well-known consequence of H\"ormander--Mikhlin estimates is the fact that this limit exists on $L^p(\R)$ and $\rz$ maps $L^p(\R)$ continuously into itself, for $1<p<\infty$.

The same holds for $H^{s,p}(\R)$ and $\hmmz(\R)$, being the former isomorphic to $L^p(\R)$ via $v\mapsto\mathcal{F}^{-1}[(1+4\pi^2\abs{\xi}^2)^{s/2}\mathcal{F}v]$ and the latter to $L^2(\R)$ via $v\mapsto\mathcal{F}^{-1}((2\pi\abs{\xi})^{-1/2}v)$.

Moreover, $\rz$ also maps $\hardy(\R)$ continuously into itself: this follows from \cite[Corollary~2.4.7]{Gra14M} and $\rz(\rz v)=-v$ for $v\in L^1(\R)$.

\subsection{Littlewood--Paley decomposition} \label{parasec}

We briefly recall a well-known tool in harmonic analysis, the  Littlewood--Paley  dyadic  decomposition.
Such a decomposition can be obtained as follows.
Let $\chi\in C^\infty_c(B(0,2))$ be an even function, with $\chi=1$ on $B(0,1)$.
Let $\varrho:=\chi-\chi(2\cdot)$ and observe that the support of $\varrho$ is included in the annulus $B(0,2)\setminus B(0,1/2)$.

Let $\varrho_0:=\chi$ and $\varrho_j=\varrho(2^{-j}\cdot)$ for $j>0$, so that the support of $\varrho_j$, for $j>0$, is contained in $B(0,2^{j+1})\setminus B(0,2^{j-1})$. The functions $(\rho_j)_{j\in\mathbb N}$ realize a so-called \emph{inhomogeneous} dyadic partition of unity, i.e. $\sum_{j=0}^\infty\rho_j=1$ pointwise.
We further denote $\chi_j(\xi):=\sum_{k=0}^j\varrho_k=\chi(2^{-j}\cdot)$.

For every $v\in{\mathcal{S}}'(\R)$  we define the inhomogeneous Littlewood--Paley projection operators:
\begin{align}\label{lpprojections}
	v_j=\mathcal{F}^{-1}[\varrho_j \hat{v}],\quad v^j=\mathcal{F}^{-1}[\chi_j \hat{v}].
\end{align}
Roughly, $v_j$ and $v^j$ mimic a frequency projection to the annulus $B(0,2^j)\setminus B(0,2^{j-1})$ and to the ball $B(0,2^j)$, respectively.

We observe that $v^j=\sum_{k=0}^{j} v_k$  and $v=\sum_{k=0}^{\infty}v_k$ in the distributional sense.
Given $v,w\in {\mathcal{S}}^\prime(\R)$, we can  formally split  their product in the following way:
\begin{equation}\label{decompbis}
vw=\Pi_1(v,w)+\Pi_2(v,w)+\Pi_3 (v,w),\end{equation}
where
\begin{align*}
	\Pi_1(v,w):=\sum_{j=3}^{+\infty} v_j w^{j-3},\quad\Pi_2(v,w):=\sum_{j=3}^{+\infty}  v^{j-3}w_j,\quad
	\Pi_3(v,w):=\sum_{j=0}^{\infty}  v_j\sum_{\abs{k-j}<3} w_k.
\end{align*}

We observe that the support ${\mathcal{F}}[v_jw^{j-3}]$ is contained in the sum of the supports of $\mathcal{F} v_j$ and $\mathcal{F} w^{j-3}$, i.e. in the annulus $B(0,2^{j+2})\setminus B(0,2^{j-2})$ (for $j\ge 3$). A similar remark applies to $\mathcal{F}[v^{j-3}w_j]$.

Next we recall the definition of the inhomogeneous Besov spaces ${B}_{p,q}^s(\R)$  and  Triebel--Lizorkin spaces ${F}_{p,q}^s(\R)$ in terms of the above dyadic decomposition.
\begin{definition}
	Let $s\in\R$ and $1\le p,q\le\infty$. For $f\in{\mathcal{S}}^\prime (\R^n)$ we set 
	\[ \begin{aligned}
	&\|v\|_{{B}_{p,q}^s}:=\Bigg(\sum_{j=0}^{\infty}2^{jsq}\|v_j\|_{L^{p}}^q\Bigg)^{1/q}&&\text{if }q<\infty, \\
	&\|v\|_{{B}_{p,q}^s}:=\sup_{j\in \mathbb N} 2^{js}\|v_j\|_{L^{p}}&&\text{if }q=\infty.
	\end{aligned} \]
	When $1\le p,q<\infty$ we also set
	\[ \|v\|_{{F}_{p,q}^s}=\Bigg\|\Bigg(\sum_{j=0}^{\infty}2^{jsq}|v_j|^q\Bigg)^{1/q}\Bigg\|_{L^p}. \]
\end{definition}
The space of all $v\in\mathcal{S}'(\R)$ for which $\|v\|_{{B}_{p,q}^s}<\infty$ is the inhomogeneous Besov space with indices 
$s,p,q$ and is denoted by ${B}_{p,q}^s(\R)$. The space of all $v\in\mathcal{S}'(\R)$ for which $\|v\|_{F_{p,q}^s}<\infty$ is the inhomogeneous
Triebel--Lizorkin space with indices 
$s,p,q$ and is denoted by $F_{p,q}^s({\R})$.

A well-known fact is that $H^{s,p}(\R)=F_{p,2}^s(\R)$, with equivalent norms: see e.g. \cite[Theorem~2.5.6]{Tri83}.

\begin{corollary}\label{hspholder}
	If $s>\frac{1}{p}$, then $H^{s,p}(\R)\subseteq L^\infty(\R)\cap C^{k,\alpha}(\R)$, for all $k\in\mathbb N$ and $0<\alpha<1$ with $k+\alpha\le s-\frac{1}{p}$.
\end{corollary}

\begin{proof}
	By \cite[Theorem~2.3.8]{Tri83} we can assume $k=0$, as well as $s=\alpha+\frac{1}{p}$. Setting $\tilde\varrho_j:=\varrho_{j-1}+\varrho_j+\varrho_{j+1}$ (with $\varrho_{-1}:=0$), we have $v_j=\mathcal{F}^{-1}(\tilde\varrho_j\mathcal{F}v_j)$ and $\norm{\mathcal{F}^{-1}\tilde\varrho_j}_{L^{p'}}\ls 2^{j/p}$, $\norm{\nabla(\mathcal{F}^{-1}\tilde\varrho_j)}_{L^{p'}}\ls 2^{j+j/p}$ (as $\mathcal{F}^{-1}\tilde\varrho_j=2^{j-2}(\mathcal{F}^{-1}\tilde\varrho_2)(2^{j-2}\cdot)$ for $j\ge 2$). Hence, given $0<h<1$,
	\begin{align*} &\sum_{j=0}^\infty\norm{v_j}_{L^\infty}=\sum_{j=0}^\infty\norm{(\mathcal{F}^{-1}\tilde\varrho_j)*v_j}_{L^\infty}\ls\sum_{j=0}^\infty 2^{j/p}\norm{v_j}_{L^p}\le\sum_{j=0}^\infty 2^{-j(s-1/p)}\norm{v}_{F_{p,2}^s}\ls\norm{v}_{H^{s,p}}, \\
	&\sum_{j=0}^\infty\norm{v_j(\cdot+h)-v_j}_{L^\infty}\ls\sum_{2^jh\le 1}h\norm{\nabla v_j}_{L^\infty}+\sum_{2^jh>1}\norm{v_j}_{L^\infty}\le\sum_{2^jh\le 1}h\norm{\nabla(\mathcal{F}^{-1}\tilde\varrho_j)*v_j}_{L^\infty} \\
	&+\sum_{2^jh>1}\norm{(\mathcal{F}^{-1}\tilde\varrho_j)*v_j}_{L^\infty}\ls\Bigg(\sum_{2^jh\le 1}h2^{j(1+1/p-s)}+\sum_{2^jh>1}2^{-j(s-1/p)}\Bigg)\norm{v}_{F_{p,2}^s}\ls h^\alpha\norm{v}_{H^{s,p}}. \qedhere \end{align*}
\end{proof}

Similarly, one can form the \emph{homogeneous} Littlewood--Paley decomposition using instead $\varrho_j:=\varrho(2^{-j}\cdot)$ and $\chi_j:=\chi(2^{-j}\cdot)$, for all $j\in\Z$, and defining $v_j$ and $v^j$ as in \eqref{lpprojections}. One then has the formal identities
\[ v=\sum_{j\in\Z}v_j,\quad v^j=\sum_{k\le j}v_k,\quad vw=\Pi_1(v,w)+\Pi_2(v,w)+\Pi_3(v,w), \]
but notice that not even the first two are always true distributionally: for instance they fail when $v=1$ (in which case $v_k=0$ for all $k\in\Z$). This reflects the fact that $\sum_{j\in\Z}\varrho_j=\uno_{\R\setminus\set{0}}$ and $\sum_{k\le j}\varrho_k=\uno_{\R\setminus\set{0}}\chi_j$. Using this homogeneous decomposition, with the same formulas as above one can define the homogeneous Besov and Triebel--Lizorkin spaces
 $\dot{B}^s_{p,q}(\R)$ and $\dot{F}_{p,q}^s(\R)$ (the above norms now become merely seminorms).
 
If $v\in L^p(\R)$ and $1<p<\infty$, then $\norm{v}_{L^p}\ls\norm{v}_{\dot F_{p,2}^0}$ and $\norm{v}_{\dot F_{p,2}^0}\ls\norm{v}_{L^p}$: see \cite[Theorem~6.1.2]{Gra14C}.
\par

\subsection{Spaces on the unit circle \texorpdfstring{$\mathcal{S}^1$}{S\^{}1}}\label{circle}

We let $\mathcal{D}(\mathcal S^1):=C^\infty(\mathcal{S}^1)$ be the Fr\'echet space of smooth functions on $\mathcal S^1=\R/2\pi\Z$ and $\mathcal{D}'(\mathcal{S}^1)$ its topological dual.
The product of two elements in $\mathcal{D}'(\mathcal{S}^1)$ is defined as before for $\R$.
For $v\in\mathcal{D}'(\mathcal S^1)$ and $k\in\Z$ we let $\hat v(k):=\frac{1}{2\pi}\ang{v,e^{-inx}}$.

Notice that, for all $v\in\mathcal{D}'(\mathcal{S}^1)$, there exists some $N>0$ such that $\abs{\hat v(k)}\ls(1+\abs{k})^N$. Also, we recall that $v\in C^\infty(\mathcal S^1)$ if and only if the Fourier coefficients $\hat v(k)$ have rapid decay, i.e. $\sup_k(1+\abs{k})^N\abs{\hat v(k)}<\infty$ for all $N>0$.

Given $s\in\R$, we define the Sobolev space
\[ H^s(\mathcal S^1):=\set{v\in\mathcal{D}'(\mathcal{S}^1):\norm{v}_{H^s}^2:=\sum_{k\in\Z}(1+\abs{k}^2)^s\abs{\hat v(k)}^2<\infty}. \]
We observe that $\mathcal{D}'(\mathcal S^1)=\bigcup_{s\in\R}H^s(\mathcal S^1)$. Also, the Fr\'echet space structure of $\mathcal{D}(\mathcal S^1)$ is equivalent to the one given by all $H^s$-norms with $s\in\mathbb N$, by the embeddings $C^{s}(\mathcal S^1)\subseteq H^s(\mathcal S^1)\subseteq C^{s-1}(\mathcal S^1)$. Hence, by the uniform boundedness principle, any sequence $v_j$ converging in $\mathcal{D}'(\mathcal S^1)$ will form a bounded set in some $H^{-s}(\mathcal S^1)$, with $s\in\mathbb N$ (by the canonical duality with $H^s(\mathcal S^1)$).

\begin{lemmaen}
	The space $H^{1/2}(\mathcal S^1)$ equals the set of traces of $H^1(\mathbb{D})$. Moreover, for $v\in L^2(\mathcal S^1)$
	\begin{equation}\label{hmzs1} \iint_{(\mathcal S^1)^2}\frac{\abs{v(e^{i\theta})-v(e^{i\tau})}^2}{\abs{e^{i\theta}-e^{i\tau}}^2}\,d\theta\,d\tau=2\pi\sum_{k\in\Z}\abs{k}\abs{\widehat{v}(k)}^2. \end{equation}
\end{lemmaen}

\begin{proof}
	Given $u\in C^\infty(\obar{\mathbb{D}})$, let $v:=\restr{u}{\mathcal S^1}$ be its trace and
	\[ \tilde v(re^{i\theta}):=\sum_{k\in\Z}\widehat{v}(k)r^{\abs{k}}e^{ik\theta}, \]
	which lies in $C^\infty(\obar{\mathbb{D}})$, is harmonic and has trace $v$. We have $\int_{\mathbb{D}}\ang{\nabla\tilde v,\nabla(u-\tilde v)}=0$ by the divergence theorem, so
	\[ \int_{\mathbb{D}}\abs{\nabla u}^2=\int_{\mathbb{D}}\abs{\nabla\tilde v}^2+2\int_{\mathbb{D}}\ang{\nabla\tilde v,\nabla(u-\tilde v)}
	+\int_{\mathbb{D}}\abs{\nabla(u-\tilde v)}^2\ge\int\abs{\nabla \tilde v}^2. \]
	A straightforward computation shows that the last integral equals $2\pi\sum_{k\in\Z}\abs{k}\abs{\widehat{v}(k)}^2$, so by density of smooth functions we deduce that the trace of a function in $H^1(\mathcal{D})$ lies in $H^{1/2}(\mathcal S^1)$. Conversely, given $v\in H^{1/2}(\mathcal S^1)$ one checks that $\tilde v$, with the above definition, is in $H^1(\mathbb{D})$. It has trace $v$ since $\tilde v\in C^\infty(\mathbb{D})$ and, as $\tau\uparrow 1$, $\tilde v(\tau\cdot)\to\tilde v$ in $H^1(\mathbb{D})$, as well as $\restr{v(\tau\cdot)}{\mathcal S^1}\to v$ in $L^2(\mathcal S^1)$. Finally, the left-hand side of \eqref{hmzs1} equals
	\[ \int_{\mathcal S^1}\frac{\norm{v-v(e^{i\sigma}\cdot)}_{L^2}^2}{\abs{1-e^{i\sigma}}^2}\,d\sigma
	=2\pi\int_{\mathcal S^1}\frac{\sum_k\abs{\widehat{v}(k)}^2\abs{1-e^{ik\sigma}}^2}{\abs{1-e^{i\sigma}}^2}
	=2\pi\sum_k\abs{\widehat{v}(k)}^2\int_{\mathcal S^1}\abs{\sum_{\ell=0}^{\abs{k}-1}e^{i\ell\sigma}}^2\,d\sigma. \qedhere \]

\end{proof}

Given a smooth compact Riemannian surface $(S,g)$ with boundary, one can define the spaces $H^s(\de S)$ by isometrically identifying each boundary component with (a dilation of) $\mathcal{S}^1$.
The last lemma, together with a partition of unity argument, can be used to show that the traces of functions in $H^1(S)$ are precisely the functions in $H^{1/2}(\de S)$. In particular, each $v\in H^{1/2}(\de S)$ has a unique harmonic extension $\tilde v\in H^1(S)$.

\section{Commutator estimates}\label{tcsec}

We introduce the following commutators for functions defined on the real line:
\begin{align*} T(Q,v)&:=\ql(Qv)+(\ql Q)v-Q(\ql v), \\
U(Q,v)&:=-\rz\ql(Qv)+(\rz\ql Q)v+Q(\rz\ql v), \\
T^*(P,Q)&:=(\ql P)Q+P(\ql Q)-\ql(PQ), \\
U^*(P,Q)&:=(\rz\ql P)Q+P(\rz\ql Q)-\rz\ql(PQ), \\
\Lambda(Q,v)&:=Qv+\rz(Q\rz v), \\
F(f,v)&:=\rz f\rz v-fv.
\end{align*}
The notation $T^*$ and $U^*$ is motivated by the formal identities
\[ \int PT(Q,v)=\int T^*(P,Q)v,\quad \int PU(Q,v)=\int U^*(P,Q)v. \]

Using the technology of Littlewood--Paley decomposition and paraproducts, one can establish the following estimates of integrability by compensation.

\begin{thm}\label{tustarest}
	If $P,Q\in\hmzb(\R)$, we have $T^*(P,Q),U^*(P,Q)\in L^{2,1}(\R)$ and
	\[ \norm{T^*(P,Q)}_{L^{2,1}},\norm{U^*(P,Q)}_{L^{2,1}}\ls\norm{P}_\hmz\norm{Q}_\hmz. \]
\end{thm}

\begin{proof}
	By \cite[Theorem~1.7]{DLR09} we have $\ql T^*(P,Q)\in\hardy(\R)$, with
	\[ \bnorm{\ql T^*(P,Q)}_\hardy\ls\norm{P}_\hmz\norm{Q}_\hmz. \]
	The estimate for $T^*$ follows from the fact that $\mql(\hardy(\R))\subseteq L^{2,1}(\R)$. The estimate for $U^*$ can be obtained in a completely analogous way. It can also be deduced from Theorem \ref{lambdaest} below, since
	\[ U^*(P,Q)=\rz T^*(P,Q)+\Lambda(P,\rz\ql Q)+\Lambda(Q,\rz\ql P) \]
	and $\rz$ maps the spaces $L^2(\R)$ and $L^{2,1}(\R)$ into themselves continuously.
\end{proof}

\begin{thm}\label{tuest}
	If $Q\in\hmzb(\R)$ and $v\in L^2(\R)$, we have $T(Q,v),U(Q,v)\in\hardy(\R)$ and
	\[ \norm{T(Q,v)}_\hardy,\norm{U(Q,v)}_\hardy\ls\norm{Q}_\hmz\norm{v}_{L^2}. \]
\end{thm}

\begin{proof}
	For the estimate of $T(P,Q)$, we refer the reader to the proof of \cite[Theorem~1.3]{DaL15} (where one just replaces $\ql u$ with $v$). The estimate of $U(Q,v)$ can be achieved with a completely analogous proof. It also follows from the identity
	\[ U(Q,v)=-T(Q,\rz v)-F(\ql Q,\rz v)+\ql\Lambda(Q,\rz v) \]
	and Theorem \ref{fest}, together with the estimate $\bnorm{\ql\Lambda(Q,\rz v)}_\hardy\ls\norm{Q}_\hmz\norm{v}_{L^2}$ (see the proof of Theorem \ref{lambdaest}).
\end{proof}

The two following results now follow from Theorems \ref{tustarest} and \ref{tuest} by a duality argument.

\begin{corollary}\label{tustarinfty}
	If $P,Q\in\hmzb(\R)$, we have
	\[ \norm{T^*(P,Q)}_{L^2},\norm{U^*(P,Q)}_{L^2}\ls\norm{P}_\hmz\bnorm{\ql Q}_{L^{2,\infty}}. \]
\end{corollary}

\begin{proof}
	Since $T(P,Q)$ vanishes if $P$ or $Q$ is constant,  we can assume that $P,Q\in \mathcal{S}(\R)$ (see Lemma \ref{hmzapprox} and Remark \ref{linftyapprox}). For any $v\in\mathcal{S}(\R)$
	\begin{align*} \int T^*(P,Q)v&=\int T^*(Q,P)v=\int QT(P,v)=\int \ql Q \mql T(P,v) \\
	&\ls\bnorm{\ql Q}_{L^{2,\infty}}\bnorm{\mql T(P,Q)}_{L^{2,1}}\ls\bnorm{\ql Q}_{L^{2,\infty}}\norm{T(P,Q)}_{\hardy} \\
	&\ls\bnorm{\ql Q}_{L^{2,\infty}}\norm{P}_\hmz\norm{v}_{L^2}, \end{align*}
	where we used Theorem \ref{tuest} and the fact that $\mql(\hardy(\R))\subseteq L^{2,1}(\R)$. A similar argument applies for $U^*$.
\end{proof}

\begin{corollary}\label{tuinfty}
	If $Q\in\hmzb(\R)$ and $v\in L^2(\R)$, we have $T(Q,v),U(Q,v)\in\hmmz(\R)$ and
	\[ \norm{T(Q,v)}_\hmmz,\norm{U(Q,v)}_\hmmz\ls\norm{Q}_\hmz\norm{v}_{L^{2,\infty}}. \]
\end{corollary}

\begin{proof}
	Since $T(Q,v)$ vanishes when $Q$ is constant,   we can assume that $Q,v\in\mathcal{S}(\R)$. For any $P\in\mathcal{S}(\R)$ we get
	\[ \int PT(Q,v)=\int T^*(P,Q)v\ls\norm{T^*(P,Q)}_{L^{2,1}}\norm{v}_{L^{2,\infty}}\ls\norm{P}_\hmz\norm{Q}_\hmz\norm{v}_{L^{2,\infty}}, \]
	thanks to Theorem \ref{tustarest}. A similar argument applies for $U$.
\end{proof}

\begin{thm}\label{lambdaest}
	If $Q\in\hmzb(\R)$ and $v\in L^2(\R)$, we have $\Lambda(Q,v)\in L^{2,1}(\R)$ and
	\[ \norm{\Lambda(Q,v)}_{L^{2,1}}\ls\norm{Q}_\hmz\norm{v}_{L^2}. \]
\end{thm}

\begin{proof}
	By \cite[Lemma~B.5]{DLR11} we know that $\ql\Lambda(Q,v)\in\hardy(\R)$, with
	\[ \bnorm{\ql\Lambda(Q,v)}_\hardy\ls\norm{Q}_\hmz\norm{v}_{L^2}, \]and thus $\norm{\Lambda(Q,v)}_{L^{2,1}}=\norm{\mql\ql\Lambda(Q,v)}_{L^{2,1}}\ls\norm{Q}_\hmz\norm{v}_{L^2}$ (the statement of \cite[Lemma~B.5]{DLR11} contains a wrong sign and, correspondingly, in its proof the correct dual operator of $\rz\ql$ is $-\rz\ql$ instead of $\rz\ql$, but the proof is otherwise correct).
\end{proof}

The following inequality is due to Coifman--Rochberg--Weiss.

\begin{thm}\label{fest}
	If $f,v\in L^2(\R)$, we have $F(f,v)\in\hardy(\R)$ and
	\[ \norm{F(f,v)}_\hardy\ls\norm{f}_{L^2}\norm{v}_{L^2}. \]
\end{thm}

\begin{proof}
	The Hilbert--Riesz transform $\rz$ satisfies the identity $\rz(fv-\rz f\rz v)=f\rz v+v\rz f$: indeed, taking the Fourier transform at $\xi\in\R$, this amounts to say that
	\[ -i\sgn(\xi)\int(1+\sgn(\xi-\zeta)\sgn(\zeta))\hat{f}(\xi-\zeta)\hat{v}(\zeta)\,d\zeta=-i\int(\sgn(\xi-\zeta)+\sgn(\zeta))\hat f(\xi-\zeta)\hat{v}(\zeta). \]
	If $\xi>0$, $1+\sgn(\xi-\zeta)\sgn(\zeta)-\sgn(\xi-\zeta)-\sgn(\zeta)
	=(1-\sgn(\xi-\zeta))(1-\sgn(\zeta))$ vanishes identically (since either $\zeta>0$ or $\xi-\zeta>0$). On the other hand, if $\xi<0$, $1+\sgn(\xi-\zeta)\sgn(\zeta)+\sgn(\xi-\zeta)+\sgn(\zeta)=(1+\sgn(\xi-\zeta))(1+\sgn(\zeta))$ vanishes also identically (since either $\zeta<0$ or $\xi-\zeta<0$). In both cases we get
	\[ \sgn(\xi)(1+\sgn(\xi-\zeta)\sgn(\zeta))=\sgn(\xi-\zeta)+\sgn(\zeta) \]
	and the identity follows. Thus we have $\norm{F(f,v)}_{L^1}\ls\norm{f}_{L^2}\norm{v}_{L^2}$ and
	\[ \rz F(f,v)=-f\rz v-v\rz f\in L^1(\R),\quad\norm{\rz F(f,v)}_{L^1}\ls\norm{f}_{L^2}\norm{v}_{L^2}. \]
	The thesis follows from \cite[Theorem~2.4.6]{Gra14M}.
\end{proof}

\section{H\"older continuity of \texorpdfstring{$\mz$}{1/2}-harmonic maps}\label{holdereg}
In this section we obtain the H\"older continuity of $\mz$-harmonic maps on $\de S$ with values into (at least) $C^2$-smooth closed manifolds.
 
\begin{thm}\label{regulH}
	Let $\N\subset\R^m$ be a $C^k$-smooth closed embedded manifold, with $k\ge 2$, and let $u\in {H}^{1/2}(\de S,\N)$ be $\mz$-harmonic. Then $u$ is H\"older continuous.
\end{thm} 

The strategy of proof of Theorem \ref{regulH} is similar to the one used to get for 
the H\"older continuity of of $\mz$-harmonic maps defined in $\R$ (see \cite{DLR09, DLR11, Sch12}). We provide here the details for the reader's convenience.
The proof can be described (roughly speaking) by the following steps.

{\bf 1.}  By means of a stereographic projection we can reduce to a problem on $\R$, as it was already observed in \cite{DaL15,DaL17}.

{\bf 2.} We rewrite the Euler equation in $\R$ as a Schr\"odinger-type linear system with antisymmetric potential satisfied by $(-\Delta)^{1/4} w$ (where  $w:=u\circ\psi_\ell^{-1}\circ\Pi^{-1}$, $\Pi^{-1}$ being the inverse of the stereograhic projection.

{\bf 3.} We show that $(-\Delta)^{1/4}w\in L^{p}_{loc}(\R)$ for every $p\ge 1$, which implies  that $ u\in \bigcap_{0<\delta<1/2} C^{0,\delta}_{loc}.$\par
 
{\bf 4.} We show that $w\in H^{1,p}_{loc}(\R)$ for some $p>2$, which implies that $ u\in   C^{0,\delta}_{loc}(\R)$ for some $0<\delta<1$.

In Section \ref{higherreg} we will show that this holds for all $1<p<\infty$ and all $0<\delta<1$: actually we will obtain that $u\in\bigcap_{0<\delta<1}C^{k-1,\delta}(\de S)$.
 
\begin{lemmaen}\label{circletoline}
	Let $u\in H^{1/2}(\de S,\subman)$ be a $\mz$-harmonic map and let
	$\Pi:\mathcal S^1\setminus\{i\}\to\R$ be the stereographic projection. Then $w:=u\circ\psi_\ell^{-1}\circ\Pi^{-1}\in\hmz(\R,\subman)$ satisfies
	\begin{align}
		&P^T(w)\ml w+\frac{2}{1+x^2}P^T(w)\pa{R_\ell((f_j)_{j=1}^k)\circ\Pi^{-1}}=0 \qquad\text{in }\mathcal{D}'(\R),\label{harmeq}\\
		&P^N(w)\nabla w=0\qquad\text{in }\mathcal{D}'(\R).\label{horiz}
	\end{align}
\end{lemmaen}

\begin{proof}\par
\underline{Step 1.}  We first prove \eqref{harmeq}.

\underline{Claim:} $w\in\hmz(\R,\R^m)$ and $\ml w=\frac{2}{1+x^2}(\ml(w\circ\Pi))\circ\Pi^{-1}$ in distributional sense.

\underline{Proof  of  the claim:} let $\mathbb{D}:=\set{\abs{z}<1}$ and $\mathbb{H}:=\set{\Im z>0}$ be the standard unit disk and upper half-plane in $\C$ and notice that the map
\[ \tilde\Pi:\mathbb{D}\to\mathbb{H},\quad\tilde\Pi(z):=\pa{\frac{2}{z-i}-i} \]
is conformal, with trace $\Pi$ on $\mathcal S^1\setminus\set{i}$. Hence, by conformal invariance of the Dirichlet energy, this map gives a bijection between $H^1(\mathbb{D})$ and $\dot H^1(\mathbb{H}):=\{w\in W^{1,2}_{loc}(\mathbb{H}):\int_\mathbb{H}\abs{\nabla w}^2dx <\infty\}$.
Moreover, $\Pi$ gives a bijection between $H^{1/2}(\mathcal S^1)$ and $\hmz(\R)$: indeed, for a real measurable function $f$ on $\R$,
\begin{equation}\label{circorettanorme} \iint_{\R^2}\frac{\abs{f(x)-f(y)}^2}{\abs{x-y}^2}\,dx\,dy=\iint_{(\mathcal S^1)^2}\frac{\abs{f\circ\Pi(e^{i\theta})-f\circ\Pi(e^{i\tau})}^2}{\abs{e^{i\theta}-e^{i\tau}}^2}\,d\theta\,d\tau, \end{equation}
since $\abs{\Pi'(e^{i\theta})}=\frac{2}{\abs{e^{i\theta}-i}^2}$
and $\abs{\Pi(e^{i\theta})-\Pi(e^{i\tau})}^{-2}=\frac{\abs{e^{i\theta}-i}^2\abs{e^{i\tau}-i}^2}{4\abs{e^{i\theta}-e^{i\tau}}^2}$. In particular we get that  $w\in\hmz(\R,\R^m)$.
We infer that $\hmz(\R)$ is precisely the image of the trace of $\dot H^1(\mathbb{H})$ and that any $f\in\hmz(\R)$ is the trace of a unique harmonic map in $\dot H^1(\mathbb{H})$ (since the corresponding statements for the unit disk hold).

Given any $f\in C^\infty(\mathcal S^1)$, the normal derivative of its harmonic extension $\tilde f\in H^1(\mathbb{D},\R^m)$ at the boundary is given by $\frac{\de\tilde f}{\de\nu}=\ml f$, as is readily checked using the formula $\tilde f(re^{i\theta})=\sum_{n\in\Z}\hat f(n)r^{\abs n}e^{in\theta}$.
The same formula also shows that $\big\|\tilde f\big\|_{H^1(\mathbb{D})}=\norm{\ql f_\ell}_{L^2}$.
	
By Lemma \ref{hmzapprox}, $w$ can be approximated in $\mathcal{S}'(\R,\R^m)$ by a sequence $w_n=h_n+c_n\in\mathcal{S}(\R,\R^m)+\R^m$ such that $w_n\to w$ in $\hmz(\R,R^m)$ and in $\mathcal{S}'(\R,\R^m)$. The functions $f_n:=w_n\circ\Pi$ extend smoothly to all the circle. By conformality of $\tilde\Pi$, $\tilde w_n:=\tilde f_n\circ\tilde\Pi^{-1}$ is the unique harmonic extension of $w_n$ in $\dot H(\mathbb{H})$ and its normal derivative is
\[ \frac{\de\tilde w_n}{\de\nu}=\abs{\Pi'\circ\Pi^{-1}}^{-1}\frac{\de\tilde f_n}{\de\nu}\circ\Pi^{-1}=\frac{2}{x^2+1}\frac{\de\tilde f_n}{\de\nu}\circ\Pi^{-1}. \]
By uniqueness, $\tilde w_n(x+iy)=\int_{\R} e^{-2\pi y\abs{\xi}}e^{2\pi ix\xi}\,\hat h_n(\xi)\,d\xi$
and thus $\frac{\de\tilde w_n}{\de\nu}(x)=(-\Delta)^{1/2}w_n$.

From \eqref{circorettanorme} and \eqref{hmzs1}, $\ql f_n\to\ql(w\circ\Pi)$ in $L^2(\mathcal S^1,\R^m)$. Hence,
	\begin{align*} \ml w&=\lim_{k\to\infty}\ml w_n=\lim_{k\to\infty}\frac{\de\tilde w_n}{\de\nu}=\lim_{n\to\infty}\frac{2}{x^2+1}\frac{\de\tilde f_n}{\de\nu}\circ\Pi^{-1} \\
	&=\lim_{k\to\infty}\frac{2}{x^2+1}(\ml f_n)\circ\Pi^{-1}=\frac{2}{x^2+1}(\ml(w\circ\Pi))\circ\Pi^{-1} \end{align*}
	in the distributional sense.
	Therefore we can conclude that \eqref{harmeq} holds.\par

	\underline{Step 2.} Next we show \eqref{horiz}. To this aim let us fix a nonnegative bump function $\rho\in C^\infty_c(B(0,1))$ with $\int\rho=1$ and let $w_\epsilon:=\rho_\epsilon*w$, where $\rho_\epsilon:=\epsilon^{-1}\rho(\epsilon^{-1}\cdot)$. From \eqref{hmzfourier} it immediately follows that $w_\epsilon\to w$ in $\hmz(\R,\R^m)$, i.e.
	\begin{equation}\label{incrconv} \frac{w_\epsilon(x)-w_\epsilon(y)}{\abs{x-y}}\to\frac{w(x)-w(y)}{\abs{x-y}}\qquad\text{in }L^2(\R^2,\R^m). \end{equation}
	In particular, for some sequence $\epsilon_j\downarrow 0$ there exists $h\in L^2(\R^2)$ such that $\frac{\abs{w_{\epsilon_j}(x)-w_{\epsilon_j}(y)}}{\abs{x-y}}\le h(x,y)$ and $w_{\epsilon_j}\to w$ a.e. Moreover, since $\subman$ is a $C^2$ submanifold, there exists a neighborhood $U\supseteq\subman$ such that the map $p\in C^1(U,\subman)$, associating to $x\in U$ the unique nearest point $p(x)$ on $\subman$, is defined. Notice that
	$\dist(w_\epsilon,\subman)\to 0$ in $L^\infty(\R)$, as
	\begin{align*} \dist(w_\epsilon(x),\subman)^2&\le\int\abs{w_\epsilon(x)-w(x-z)}^2\rho_\epsilon(z)\,dz
	\le\iint\abs{w(x-y)-w(x-z)}^2\rho_\epsilon(y)\rho_\epsilon(z)\,dy\,dz \\
	&\ls\epsilon^{-2}\iint_{B(0,\epsilon)^2}\abs{w(x-y)-w(x-z)}^2\,dy\,dz
	\ls\iint_{B(x,\epsilon)^2}\frac{\abs{w(y)-w(z)}^2}{\abs{y-z}^2}\,dy\,dz, \end{align*}
	which converges to $0$ uniformly in $x$. Thus, eventually $p(w_{\epsilon_j})\in\hmz(\R,\subman)$ is defined. Since $P^N\circ p(w_{\epsilon_j})\nabla(p(w_{\epsilon_j}))=0$, it suffices to show that
	\[ P^N\circ p(w_{\epsilon_j})\to P^N\circ p(w)=P^N(w),\quad p(w_{\epsilon_j})\to p(w)=w \]
	in $\hmz(\R,\R^m)$. This immediately follows from \eqref{incrconv} by dominated convergence, since the maps $P^N\circ p$ and $p$ are Lipschitz (up to shrinking $U$).\par
	We finally remark that   $h:=-\frac{2}{1+x^2}P^T(w)\pa{R_\ell((f_j)_{j=1}^k)\circ\Pi^{-1}}$ lies in $L^1\cap L^\infty(\R,\R^m)$.
\end{proof}

 In the case  $w\in\hmz(\R,\subman)$, the quantitity $P^N{\ql w}$ enjoys special regularity properties. This has already been observed in 
  \cite{DLLR16,MS17}.
 
\begin{lemmaen}\label{PNprop}
	For any $w\in\hmz(\R,\subman)$ it holds
	\[ \abs{P^N\ql w}\ls\abs{T^*(w;w)}\quad\text{a.e.} \]
\end{lemmaen}

\begin{proof}
	Since $w$ takes values in the $C^2$ submanifold $\subman$, it holds
	\[ \abs{P^N(w(x))(w(x)-w(x+y))}\ls\abs{w(x)-w(x+y)}^2 \]
	and, in view of Lemma \ref{qlintrep}, we deduce that for some sequence $\epsilon_j\downarrow 0$
	\[ \abs{P^N(w)\ql w}(x)\ls\liminf_{j\to\infty}\int_{\R\setminus B(0,\epsilon_j)}\frac{\abs{w(x)-w(x+y)}^2}{\abs{y}^{3/2}}\,dy, \]
	\begin{align*} T^*(w;w)(x)&=\ql w\cdot w+w\cdot\ql w-\ql(w\cdot w) \\
	&=c\lim_{j\to\infty}\int_{\R\setminus B(0,\epsilon_j)}\frac{\abs{w(x)-w(x+y)}^2}{\abs{y}^{3/2}}\,dy, \end{align*}
	thanks to the identity (with $z:=x+y$)
	\[ (w(x)-w(z))\cdot w(x)+w(x)\cdot(w(x)-w(z))-(w(x)\cdot w(x)-w(z)\cdot w(z))=\abs{w(x)-w(z)}^2. \qedhere
	\]
\end{proof}

In what follows, given $x_0\in\R$ and $r>0$, we set $B:=B(x_0,r)$, $A_0:=B(x_0,2r)$ and, for $j\ge 1$, $A_j:=B(x_0,2^{j+1}r)\setminus B(x_0,2^jr)$. 
We give some preliminary estimates beforehand.
\begin{lemmaen}\label{tech2infty}
	For any $w\in\hmz(\R)$ and any $1\le p<\infty$ we have
	\begin{align}
	r^{-1/p}\norm{w-(w)_{B}}_{L^p(B)}&\ls\norm{w}_\hmz,\label{estpnorm}\\
	\begin{split}r^{-1/2}\norm{w-(w)_{B}}_{L^2(B)}&\ls r^{-3/4}\pa{\iint_{B^2}\frac{\abs{w(x)-w(y)}^2}{\abs{x-y}^{1/2}}\,dx\,dy}^{1/2} \label{halfnorm} \\
	&\ls\sum_{j=0}^\infty 2^{-j/2}\big\|\ql w\big\|_{L^{2,\infty}(A_j)}. \end{split}
	\end{align}
\end{lemmaen}

\begin{proof}
	\underline{Proof of \eqref{estpnorm}:}
	by translating and rescaling, we can assume $x_0=0$ and $r=1$.
	Moreover,  we can suppose $w=\mql v=c\abs{x}^{-1/2}*v$ for some $v\in\mathcal{S}(\R)$.
	Letting $w_1:=\mql(v\uno_{A_0})$, $w_2:=\mql(v\uno_{\R\setminus A_0})$ and using Young's inequality, the mean value theorem and H\"older's inequality,
	\begin{align*} &\norm{w-(w)_{B}}_{L^p(B)}\ls\norm{w_1}_{L^p(B)}+\sup_{x,x'\in B}\abs{w_2(x)-w_2(x')} \\
	&\ls\norm{(\abs{x}^{-1/2}\uno_{B(0,3)})*(v\uno_{A_0})}_{L^p(B)}+\sup_{x,x'\in B}\int_{\R\setminus A_0}\abs{\abs{x-y}^{-1/2}-\abs{x'-y}^{-1/2}}\abs{v(y)}\,dy \\
	&\ls\norm{\abs{x}^{-1/2}}_{L^{2p/(p+2)}(B(0,3))}\norm{v}_{L^2(A_0)}+\int_{\R\setminus A_0}\abs{y}^{-3/2}\abs{v(y)}\,dy\ls\norm{v}_{L^2} \end{align*}
	(assuming without loss of generality $p\ge 2$), which proves the first part.
	
	\underline{Proof of \eqref{halfnorm}:}
	assuming again $x_0=0$ and $r=1$, by Jensen's inequality it follows
	\begin{equation}\label{est1} \begin{aligned}\norm{w-(w)_{B}}_{L^2(B)}^2&\ls\iint_{B^2}\abs{w(x)-w(y)}^2\,dx\,dy
	\ls\iint_{B^2}\frac{\abs{w(x)-w(y)}^2}{\abs{x-y}^{1/2}}\,dx\,dy \\
	&\ls\int_0^2\int_{B}\frac{\abs{w(x+h)-w(x)}^2}{h^{1/2}}\,dx\,dh. \end{aligned}\end{equation}
	We set $f_h(z):=(\abs{z+h}^{-1/2}+\abs{z}^{-1/2})\uno_{B(0,2h)}(z)$,
	\begin{align}\label{est2} \abs{w(x+h)-w(x)}&\ls\int\abs{\abs{x+h-y}^{-1/2}-\abs{x-y}^{-1/2}}\abs{v(y)}\,dy \nonumber \\
	&\ls f_h*\abs{v}(x)+\int_{\R\setminus B(x,2h)}\abs{\abs{x+h-y}^{-1/2}-\abs{x-y}^{-1/2}}\abs{v(y)}\,dy \\
	&\ls f_h*\abs{v}(x)+h\int_{\R\setminus B(x,2h)}\abs{x-y}^{-3/2}\abs{v(y)}\,dy,\nonumber  \end{align}
	where we used again the mean value theorem.
	Notice that, by Young's inequality,
	\begin{align} \begin{aligned} \label{est3}&\int_0^2 h^{-1/2}\int_{B}|f_h*\abs{v}(x)|^2\,dx\,dh=\int_0^2 h^{-1/2}\norm{f_h*(\abs{v}\uno_{B(0,5)})}_{L^2(B)}^2\,dh \\
	&\le\int_0^2 h^{-1/2}\norm{f_h}_{L^{4/3}}^2\norm{v}_{L^{4/3}(B(0,5))}^2\,dh\ls\norm{v}_{L^{2,\infty}(B(0,5))}^2, \end{aligned} \end{align}
	since $\norm{f_h}_{L^{4/3}}\ls h^{1/4}$. On the other hand, by H\"older's inequality,
	\begin{equation}\label{est4}\int_{A_0\setminus B(x,2h)}\abs{x-y}^{-3/2}\abs{v(y)}\,dy
	\ls\Big(\int_{2h}^\infty t^{-9/2}\,dt\Big)^{1/3}\norm{v}_{L^{3/2}(A_0)}\ls h^{-7/6}\norm{v}_{L^{2,\infty}(A_0)}, \end{equation}
	while, since $\abs{x-y}^{-3/2}\ls 2^{-3j/2}$ for $x\in B$ and $y\in A_j$ (when $j\ge 1$),
	\begin{equation}\label{est5} \int_{\R\setminus A_0}\abs{x-y}^{-3/2}\abs{v(y)}\,dy
	=\sum_{j=1}^\infty\int_{A_j}\abs{x-y}^{-3/2}\abs{v(y)}\,dy
	\ls\sum_{j=1}^\infty 2^{-j}\norm{v}_{L^{2,\infty}(A_j)}.  \end{equation}
	By combining \eqref{est1}--\eqref{est5} and by applying Cauchy--Schwarz inequality we conclude that
	\begin{align*} &\int_0^2\int_{B_1}\frac{\abs{w(x+h)-w(x)}^2}{h^{1/2}}\,dx\,dh
	\ls\norm{v}_{L^{2,\infty}(B(0,5))}^2+\int_0^2   h^{-5/6}\norm{v}_{L^{2,\infty}(A_0)}^2\,dh \\
	&+\int_0^2 h^{-1/2}h^2\Bigg(\sum_{j=1}^\infty 2^{-j}\norm{v}_{L^{2,\infty}(A_j)}\Bigg)^2\,dh
	\ls\norm{v}_{L^{2,\infty}(A_0)}^2+\sum_{j=1}^\infty 2^{-j}\norm{v}_{L^{2,\infty}(A_j)}^2. \end{align*}
	The thesis follows.
\end{proof}


\begin{lemmaen}\label{loctstar}
	Given $w\in\hmzb(\R,\R^m)$, we can estimate
	\begin{align*} \norm{T^*(w;w)}_{L^2(B)}&\ls\pa{\norm{w}_{\hmz(B(x_0,4r))}+\big\|\ql w\big\|_{L^{2,\infty}(B(x_0,4r))}}\big\|\ql w\big\|_{L^{2,\infty}(A_0)} \\
	&\quad+\sum_{j=1}^\infty 2^{-j/4}\pa{\norm{w}_{\hmz(B(x_0,4r))}+\bnorm{\ql w}_{L^{2,\infty}(A_j)}}\big\|\ql w\big\|_{L^{2,\infty}(A_j)}, \end{align*}
	where $\norm{w}_{\hmz(B(x_0,4r))}^2:=\iint_{B(x_0,4r)^2}\frac{\abs{w(x)-w(y)}^2}{\abs{x-y}^2}\,dx\,dy$.
\end{lemmaen}

\begin{proof}
	Again we can assume $x_0=0$, $r=1$.  Given $\rho\in C^\infty_c(B(0,3))$ with $\rho= 1$ on $B(0,2)$, we define $w_0:=w-(w)_{B(0,4)}$
	and observe that $T^*(w;w)=T^*(w_0;w_0)$, since $T^*$ vanishes when one of the arguments is constant, while $\norm{\rho w_0}_\hmz^2$ equals
	\begin{align*} \norm{\rho w_0}_\hmz^2
	&\ls\iint_{B(0,4)}\frac{\abs{\rho(x)w_0(x)-\rho(y)w_0(y)}^2}{\abs{x-y}^2}+\iint_{B(0,4)\times(\R\setminus B(0,4))}\frac{\abs{\rho(x)w_0(x)}^2}{\abs{x-y}^2} \\
	&\ls\iint_{B(0,4)}\frac{\abs{w_0(x)-w_0(y)}^2}{\abs{x-y}^2}+2^{-s}\int_{B(0,4)}\abs{w_0-(w_0)_{B(0,4)}}^2 \\
	&\ls\iint_{B(0,4)^2}\abs{w_0(x)-w_0(y)}^2(\abs{x-y}^{-2}+2^{-2s})\ls\norm{w}_{\hmz(B(0,4))}^2, \end{align*}
	where we split $\rho(x)w_0(x)-\rho(y)w_0(y)=\rho(x)(w_0(x)-w_0(y))+(\rho(x)-\rho(y))w_0(y)$ and used the fact that $(w_0)_{B(0,4)}=0$.
	Next we write 
	\[ T^*(w;w)=T^*(\rho w_0;\rho w_0)+T^*((1-\rho)w_0;\rho w_0)+T^*(w_0;(1-\rho)w_0), \]
	so that Corollary \ref{tustarinfty} gives
	\begin{align*} \norm{T^*(w;w)}_{L^2(B)}&\ls\norm{\rho w_0}_\hmz\big\|\ql(\rho w_0)\big\|_{L^{2,\infty}}+\norm{T^*((1-\rho)w_0;\rho w_0)}_{L^2(B)} \\
	&\quad+\norm{T^*(w_0;(1-\rho)w_0)}_{L^2(B)}. \end{align*}

	\underline{Estimate of $\norm{\rho w_0}_\hmz\big\|\ql(\rho w_0)\big\|_{L^{2,\infty}}$:}
	by the above $\norm{\rho w_0}_\hmz\ls\norm{w}_{\hmz(B(0,4))}$. Also,
	\begin{equation}\label{rhow0}
	\bnorm{\ql(\rho w_0)}_{L^{2,\infty}(\R)}^2\ls\bnorm{\rho\ql w}_{L^{2,\infty}(\R)}^2+\int\abs{\int\frac{(\rho(x)-\rho(y))w_0(y)}{\abs{x-y}^{3/2}}\,dy}^2\,dx \end{equation}
	(see Lemma \ref{qlintrep}). It suffices to bound the last term of \eqref{rhow0}. Splitting $(\rho(x)-\rho(y))w_0(y)=-(\rho(x)-\rho(y))(w_0(x)-w_0(y))+(\rho(x)-\rho(y))w_0(x)$ and using Cauchy--Schwarz, as well as $\abs{\rho(x)-\rho(y)}\ls \abs{x-y}$,
	\begin{align*} &\int_{B(0,4)}\abs{\int_{B(0,4)}\frac{(\rho(x)-\rho(y))w_0(y)}{\abs{x-y}^{3/2}}\,dy}^2\,dx\ls\norm{w_0}_{L^2(B(0,4))}^2 \\
	&\quad+ \int_{B(0,4)}\pa{\int_{B(0,4)}\frac{\abs{\rho(x)-\rho(y)}^2}{\abs{x-y}^{5/2}}\,dy}\pa{\int_{B(0,4)}\frac{\abs{w(x)-w(y)}^2}{\abs{x-y}^{1/2}}\,dy}\,dx
	 \\
	&\ls \norm{w_0}_{L^2(B(0,4))}^2+\iint_{B(0,4)}\frac{\abs{w(x)-w(y)}^2}{\abs{x-y}^{1/2}}\,dx\,dy. \end{align*}
	Moreover,
	\begin{align*} \int_{\R\setminus B(0,4)}\abs{\int_{B(0,4)}-\frac{\rho(y)w_0(y)}{\abs{x-y}^{3/2}}\,dy}^2\,dx
	&\ls \int_{B(0,3)}\Bigg(\int_{\R\setminus B(0,4)}\frac{\abs{w_0(y)}^2}{\abs{x-y}^3}\,dx\Bigg)\,dy \\
	&\ls \norm{w_0}_{L^2(B(0,4))}^2. \end{align*}
	
	Now we use   the elementary inequality 
	\begin{align*} \norm{w_0}_{L^1(B(0,2^{j}))}&\ls\bnorm{w_0-(w_0)_{B(0,2^{j})}}_{L^1(B(0,2^{j}))}+\sum_{\ell=3}^{j} 2^j \big|(w_0)_{B(0,2^\ell)}-(w_0)_{B(0,2^{\ell-1})}\big|
	\\&
	\ls\sum_{\ell=2}^{j} 2^{j-\ell/2}\bnorm{w-(w)_{B(0,2^\ell)}}_{L^2(B(0,2^\ell))}
	\end{align*}
	(for $j\ge 2$) and we get 
	\begin{align*} &\pa{\int_{B(0,4)}\abs{\int_{\R\setminus B(0,4)}\frac{\rho(x)w_0(y)}{\abs{x-y}^{3/2}}\,dy}^2\,dx}^{1/2}
	\ls\sum_{j=2}^\infty 2^{-3j/2}\norm{w_0}_{L^1(A_j)} \\
	&\ls\sum_{j=2}^\infty\sum_{\ell=2}^{j+1} 2^{-j/2-\ell/2}\bnorm{w_0-(w_0)_{B(0,2^\ell)}}_{L^2(B(0,2^\ell))}\ls\sum_{\ell=s}^\infty 2^{-\ell}\bnorm{w_0-(w_0)_{B(0,4)}}_{L^2(B(0,2^\ell))}. \end{align*}
	Thus, applying Lemma \ref{tech2infty} to $B(0,4)$ and $B(0,2^\ell)$, we get 
	\begin{align*} &\bnorm{\ql(\rho w_0)}_{L^{2,\infty}}\ls \bnorm{\ql w}_{L^{2,\infty}(B(0,3))}\\
	&+\sum_{\ell=2}^\infty 2^{-\ell/2}\Bigg(\sum_{p=0}^{\ell+1}\bnorm{\ql w}_{L^{2,\infty}(A_p)}+\sum_{p=\ell+2}^\infty 2^{(\ell+1-p)/2}\bnorm{\ql w}_{L^{2,\infty}(A_p)}\Bigg) \\
	& \ls\sum_{p=0}^\infty (p+1)2^{-p/2}\bnorm{\ql w}_{L^{2,\infty}(A_p)}\ls \sum_{p=0}^\infty 2^{-p/4}\bnorm{\ql w}_{L^{2,\infty}(A_p)}. \end{align*}
	
	\underline{Estimate of  $\norm{T^*(w_0;(1-\rho)w_0)}_{L^2(B)}$:}
	by Lemma \ref{qlintrep} we have
	\begin{align*} &\bnorm{w_0\cdot\ql((1-\rho)w_0)}_{L^2(B)}\ls\norm{w_0-(w_0)_{B(0,4)}}_{L^2(B(0,4))}\bnorm{\ql((1-\rho)w_0)}_{L^\infty(B)} \\
	&\ls \norm{w}_{\hmz(B(0,4))}\sum_{j=1}^\infty 2^{-3j/2}\norm{w_0}_{L^1(A_j)}\\
	&\ls \norm{w}_{\hmz(B(0,4))}\sum_{p=0}^\infty 2^{-p/4}\bnorm{\ql w}_{L^{2,\infty}(A_p)},\end{align*}
	where the last inequality is obtained as before. Hence, $\norm{w_0\cdot\ql((1-\rho)w_0)}_{L^2(B)}$ has the desired upper bound.
	 Similarly, using Cauchy--Schwarz inequality twice,
	\begin{align*} &\bnorm{\ql((1-\rho)\abs{w_0}^2)}_{L^2(B)}\ls \bnorm{\ql((1-\rho)\abs{w_0}^2)}_{L^{\infty}(B)}\\& \ls \sum_{j=1}^\infty 2^{-3j/2}\norm{w_0}_{L^2(A_j)}^2 \\
	&\ls \sum_{j=1}^\infty 2^{-3j/2}\pa{\sum_{\ell=2}^{j+1} 2^{j/2-\ell/2}\bnorm{w_0-(w_0)_{B(0,2^\ell)}}_{L^2(B(0,2^\ell))}}^2 \\
	&\ls \sum_{j=1}^\infty\sum_{\ell=2}^{j+1} 2^{-j/2-\ell}\ell^2\bnorm{w_0-(w_0)_{B(0,2^\ell)}}_{L^2(B(0,2^\ell))}^2 \\
	&\ls \sum_{\ell=2}^\infty 2^{-3\ell/2}\ell^2\bnorm{w_0-(w_0)_{B(0,2^\ell)}}_{L^2(B(0,2^\ell))}^2 \\
	&\ls\sum_{\ell=2}^\infty 2^{-\ell/2}\ell^2\Bigg(\sum_{p=0}^{\ell+1} (p+1)^2\bnorm{\ql w}_{L^{2,\infty}(A_p)}^2+\sum_{p=\ell+2}^\infty (p+1)^2 2^{\ell-p}\bnorm{\ql w}_{L^{2,\infty}(A_p)}^2\Bigg) \\
	&\ls\sum_{p=0}^\infty 2^{-p/4}\bnorm{\ql w}_{L^{2,\infty}(A_p)}^2. \end{align*}
	 
	\underline{Estimate of $T^*((1-\rho)w_0;\rho w_0)=T^*(\rho w_0;(1-\rho)w_0)$:} analogous.
\end{proof}

\begin{lemmaen}\label{loct}
Let  $P\in\hmzb(\R)$ and $v\in L^2(\R)$. Then, uniformly in $s\ge 1$,
	\begin{align*} \bnorm{\mql T(P,v)}_{L^2(B)}&\ls\norm{v}_{L^{2,\infty}(B(0,2^s))}\sum_{j=s}^\infty 2^{s/2-j/4}\norm{P}_{\hmz(B(0,2^jr))}\\
	&\quad+\norm{P}_\hmz\sum_{j=s}^\infty 2^{-j/4}\norm{v}_{L^{2,\infty}(A_j)}. \end{align*}
\end{lemmaen}

\begin{proof}
	The proof is similar to the one of Lemma \ref{loct}, but is substantially simpler: as before
	we assume $x_0=0$, $r=1$. Setting $P_0:=P-(P)_{B(0,2^{s+3})}$, notice that $T(P,v)=T(P_0,v)$. \par Let $\rho\in C^\infty_c(B(0,2^{s+2}))$ with $\rho=1$ on $B(0,2^{s+1})$ and $\abs{\rho'}\ls 2^{-s}$. By Corollary \ref{tuinfty} we can write
	\begin{align*} \mql T(P,v)=&\mql T(\rho P_0,v\uno_{B(0,2^s)})+\mql T((1-\rho)P_0,v\uno_{B(0,2^s)}) \\
	&+\sum_{j=s}^\infty\mql T(P_0,v\uno_{A_j}) \end{align*}
	in $L^2(\R)$ and as before
	\[ \bnorm{\mql T(\rho P_0,v\uno_{B(0,2^s)})}_{L^2}\ls\norm{\rho P_0}_\hmz\norm{v}_{L^{2,\infty}(B(0,2^s))}\ls\norm{P}_{\hmz(B(0,2^{s+3}))}\norm{v}_{L^{2,\infty}(B(0,2^s))}. \]
	To estimate the two other pieces, fix any $j\ge 1$ and let
	$\chi_j\in C^\infty_c(B(0,2^{j+2}))$ with
	\[ \chi_j=1\quad\text{on }B(0,2^{j+1})\setminus B(0,(5/6)2^j),\qquad\chi_j= 0\quad\text{on }B(0,(4/6)2^j)\cup(\R\setminus B(0,2^{j+3})) \]
	and $\bnorm{\chi_j'}_{L^\infty}\ls 2^{-j}$. In particular, $\chi_j$ vanishes in a neighborhood of $\bar B$.
	Next we are going to use \cite[Proposition~2.4.8]{Gra14C} (which implies that, on $\R\setminus \bar{A_j}$, $\ql(P_0v\uno_{A_j})$ is smooth and bounded by $\abs{x}^{-3/2}*\abs{P_0v\uno_{A_j}}$) and the fact that,\footnote{Since $(P_0)_{B(0,4)}=0$, if $\ell\ge s+2$ we can write
	$P_0=P_0-(P_0)_{B(0,2^{\ell+1})}+\sum_{p=s+4}^{\ell+1}((P_0)_{B(0,2^{p})}-(P_0)_{B(0,2^{p-1})})$ and if $\ell\le s+2$ we write
	$P_0=P_0-(P_0)_{B(0,2^{\ell+1})}+\sum_{p=\ell+2}^{s+3}((P_0)_{B(0,2^{p})}-(P_0)_{B(0,2^{p-1})})$.} by Lemma \ref{tech2infty}, 
	\[ \norm{P_0}_{L^4(A_\ell)}\ls\sum_{p=0}^{\max(\ell+1,s+3)}2^{\ell/4-p/4}\bnorm{P_0-(P_0)_{B(0,2^p)}}_{L^4(B(0,2^p))}\ls\max(\ell+1,s)2^{\ell/4}\norm{P_0}_\hmz. \]
	We split $T(P_0,v\uno_{A_j})=(1-\chi_j)T(P_0,v\uno_{A_j})+\chi_jT(P_0,v\uno_{A_j})$. For all $\ell\ge 0$ we have
	\begin{align*} &\norm{(1-\chi_j)T(P_0,v\uno_{A_j})}_{L^2(A_\ell)} \\
	&\ls 2^{\ell/2}\bnorm{(1-\chi_j)\ql(P_0v\uno_{A_j})}_{L^\infty(A_\ell)}+2^{\ell/4}\norm{P_0}_{L^4(A_\ell)}\bnorm{(1-\chi_j)\ql(v\uno_{A_j})}_{L^\infty(A_\ell)} \\
	&\ls 2^{-3\max(j,\ell)/2}\pa{2^{\ell/2}\norm{P_0}_{L^4(A_j)}+2^{\ell/4}2^{j/4}\norm{P_0}_{L^4(A_\ell)}}\norm{v}_{L^{4/3}(A_j)} \\
	&\ls\max(j+1,\ell+1,s)2^{j/2+\ell/2} 2^{-3\max(j,\ell)/2}\norm{P}_\hmz\norm{v}_{L^{2,\infty}(A_j)}. \end{align*}
	Since $T(P_0,v\uno_{A_j})\in L^1(\R)$ by Theorem \ref{tuest}, it follows that for all $j\ge s$
	\begin{align*} &\bnorm{\mql((1-\chi_j)T(P_0,v\uno_{A_j}))}_{L^2(B)} \\
	&\ls\Bigg(\bnorm{\abs{x}^{-1/2}*((1-\chi_j)T(P_0,v\uno_{A_j})\uno_{A_0})}_{L^2(B)}+\sum_{\ell=1}^\infty 2^{-\ell/2}\norm{(1-\chi_j)T(P_0,v\uno_{A_j})}_{L^1(A_\ell)}\Bigg) \\
	&\ls\sum_{\ell=0}^\infty\norm{(1-\chi_j)T(P_0,v\uno_{A_j})}_{L^2(A_\ell)} \\
	&\ls \norm{P}_\hmz\norm{v}_{L^{2,\infty}(A_j)}\Bigg(\sum_{\ell=0}^j(j+1) 2^{-j+\ell/2}+\sum_{\ell=j+1}^\infty (\ell+1)2^{j/2-\ell}\Bigg) \\
	&\ls 2^{-j/4}\norm{P}_\hmz\norm{v}_{L^{2,\infty}(A_j)}. \end{align*}
	Similarly, as $T((1-\rho)P_0,v\uno_{B(0,2^s)})=\ql((1-\rho)P_0)v\uno_{B(0,2^s)}-(1-\rho)P_0\ql(v\uno_{B(0,2^s)})$,
	\begin{align*} \norm{T((1-\rho)P_0,v\uno_{B(0,2^s)})}_{L^{2,\infty}(B(0,2^s))}&\ls\sum_{p=s}^\infty 2^{-p}\norm{P_0}_{L^2(A_p)}\norm{v}_{L^{2,\infty}(B(0,2^s))} \\
	&\ls\sum_{p=s}^\infty p2^{-p/2}\norm{P}_{\hmz(B(0,2^{p+1}))}\norm{v}_{L^{2,\infty}(B(0,2^s))}, \end{align*}
	\begin{align*} \norm{T((1-\rho)P_0,v\uno_{B(0,2^s)})}_{L^2(A_\ell)}&\ls 2^{-3\ell/2}\norm{P_0}_{L^2(A_\ell)}\norm{v}_{L^1(B(0,2^s))} \\
	&\ls \ell 2^{s/2-\ell}\norm{P}_{\hmz(B(0,2^{\ell+1}))}\norm{v}_{L^{2,\infty}(B(0,2^s))} \end{align*}
	for $\ell>s$ (notice that $1-\rho$ vanishes near $\bar B(0,2^s)$). Since
	\begin{align*} &\bnorm{\mql (T((1-\rho)P_0,v\uno_{B(0,2^s)})\uno_{B(0,2^s)})}_{L^2(B)} \\
	&\ls\norm{\abs{x}^{-1/2}}_{L^{4/3}(B(0,2^{s+3}))}\norm{T((1-\rho)P_0,v\uno_{B(0,2^s)})}_{L^{4/3}(B(0,2^s))} \\
	&\ls 2^{s/2}\norm{T((1-\rho)P_0,v\uno_{B(0,2^s)})}_{L^{2,\infty}(B(0,2^s))},\end{align*}
	we get
	\[ \bnorm{\mql T((1-\rho)P_0,v\uno_{B(0,2^s)})}_{L^2(B)}\ls\sum_{j=s}^\infty 2^{s/2-j/4}\norm{P}_{\hmz(B(0,2^j))}\norm{v}_{L^{2,\infty}(A_0)}. \]
	Finally, we estimate $\norm{\mql (\chi_j T(P_0,v\uno_{A_j}))}_{L^2(B)}$ by duality: for any $\psi\in C^\infty_c(B)$ with $\norm{\psi}_{L^2}\le 1$
	\begin{align*} \big\langle\mql(\chi_j T(P_0,v\uno_{A_j})),\psi\big\rangle&=\int\chi_j T(P_0,v\uno_{A_j})\mql\psi \\
	&=\int\mql T(P_0,v\uno_{A_j})\ql(\chi_j\mql\psi). \end{align*}
	The first identity holds since $T(P_0,v\uno_{A_j})\in L^1(\R)$ (by Theorem \ref{tuest}), while the second is justified by $\chi_j\mql\psi\in C^\infty_c(\R)$, $\mql T(P_0,v\uno_{A_j})\in L^2(\R)$ (by Corollary \ref{tuinfty}) and Plancherel's theorem.
	
	We observe that, on the support of $\chi_j$, $\mql\psi$ is bounded by $2^{-j/2}$ and its derivative by $2^{-3j/2}$ (as $\mql\psi=c\abs{x}^{-1/2}*\psi$), so $f:=\chi_j\mql\psi$ has $\norm{f}_{L^\infty}\ls 2^{-j/2}$, $\norm{f'}_{L^\infty}\ls 2^{-3j/2}$ and
	\[ \bnorm{\ql f}_{L^2}^2\ls\int(2^{-j}+2^j\xi^2)\abs{\widehat f(\xi)}^2\,d\xi\ls 2^{-j}\norm{f}_{L^2}^2+2^j\norm{f'}_{L^2}^2\ls 2^{-j}. \]
	Moreover, by Corollary \ref{tuinfty}, $\norm{\mql T(P_0,v\uno_{A_j})}_{L^2}\ls\norm{P}_\hmz\norm{v}_{L^{2,\infty}(A_j)}$. We deduce that
	\[ \bnorm{\mql(\chi_j T(P_0,v\uno_{A_j}))}_{L^2}\ls 2^{-j/2}\norm{P}_\hmz\norm{v}_{L^{2,\infty}(A_j)}. \qedhere \]
\end{proof}
\subsection{Rewriting the Euler--Lagrange equation}
 
In Lemma \ref{circletoline} we have seen that $w:=u\circ\psi_\ell^{-1}\circ\Pi^{-1}\in\hmz(\R,\subman)$ satisfies 
\[ P^T(w)\ml w+\frac{2}{1+x^2}P^T(w)\pa{R_\ell((f_j)_{j=1}^k)\circ\Pi^{-1}}=0 \quad\text{in } \mathcal{D}'({\R}). \]
Therefore we can write \begin{equation}\label{eulerequ2} \ml w=P^N(w)\ml w+h \end{equation}
where $h=-\frac{2}{1+x^2}P^T(w)\pa{R_\ell((f_j)_{j=1}^k)\circ\Pi^{-1}}\in L^1\cap L^{\infty}(\R)$.
We are going to reformulate the equation \eqref{eulerequ2} in the same spirit it has been done in \cite{DLS17,MS17}.\par

This equivalent reformulation will be crucial in order to obtain the regularity of $w$. First of all, writing for simplicity $P^T$ and $P^N$ in place of $P^T(w)$ and $P^N(w)$,
\begin{align*} P^N\ml w&=\ql(P^N\ql w)+\ql P^N\ql w-T(P^N,\ql w) \\
&=\ql(P^N\ql w)+\ql P^N P^N\ql w \\
&\quad+\ql P^N P^T\ql w-T(P^N,\ql w). \end{align*}
Next, we observe that
\begin{align*} &\ql P^N P^T=-P^N\ql P^T+T^*(P^N,P^T) \\
&=P^N\ql P^N+T^*(P^N,P^T)=\Omega_0+\Omega_1+\ql P^N P^N, \end{align*}
where $\Omega_0:=P^N\ql P^N-\ql P^N P^N\in L^2(\R,\mathfrak{so}(m))$, $\Omega_1:=T^*(P^N,P^T)\in L^{2,1}(\R,\R^{m\times m})$. Hence, setting $v:=\ql w\in L^2(\R,\R^m)$, we arrive at
\begin{align}\label{schrod} \ql v=\Omega_0 v+\Omega_1 v+\ql(P^N v)+2\ql P^N(P^N v)-T(P^N,v)+h. \end{align}

\begin{thm}\label{regw}
	The map $v=\ql w$ has $\ql(P^T v),\rz\ql(P^N v)\in L^1(\R,\R^m)$ and there exists $\alpha>0$ such that
	\[ \bnorm{\ql(P^T v)}_{L^1(B(x_0,r))}+\bnorm{\rz\ql(P^N v)}_{L^1(B(x_0,r))}\ls r^\alpha, \]
	for all $r>0$, uniformly in $x_0\in\R$.
\end{thm}

\begin{proof}
	\underline{Step 1.} Fix any $x_0\in\R$. We first proceed to locally remove the antisymmetric matrix $\Omega_0$: if $R>0$ is small enough, then we can write $\Omega_0\uno_{B(x_0,R)}=\mz(Q^{-1}\ql Q-\ql Q^{-1}Q)$ for some $Q\in\hmz(\R,SO(m))$ with $\norm{Q}_\hmz\ls\norm{\Omega_0}_{L^2(B(x_0,R))}$ (see \cite[Section~4]{DLR11} and \cite[Section~4]{DLS17}). The map $\tilde v:=Qv$ then satisfies
	\begin{align*} \ql\tilde v&=Q\ql v-\ql Q v+T(Q,v) \\
	&=Q\Omega_0 v+Q\Omega_1 v+Q\ql(P^N v)+2(Q\ql P^N)(P^N v) \\
	&\quad-QT(P^N,v)+Qh-\ql Q v+T(Q,v). \end{align*}
	Using the identities
	\[ \Omega_0\uno_{B(x_0,R)}-Q^{-1}\ql Q=-\mz T^*(Q^{-1},Q), \]
	\[ Q\ql(P^N v)=\ql(QP^N v)+\ql Q P^N v-T(Q,P^N v), \]
	\[ QT(P^N,v)=T^*(Q,P^N)v+T(QP^N,v)-T(Q,P^N v), \]
	we get
	\begin{align*} \ql\tilde v&=Q\Omega_0\uno_{\R\setminus B(x_0,r)}v-\frac{Q}{2}T^*(Q^{-1},Q)v+Q\Omega_1 v+\ql(QP^N v) \\
	&\quad+\ql Q P^N v+2(Q\ql P^N)(P^N v)-T^*(Q,P^N)v-T(QP^N,v) \\
	&\quad+Qh+T(Q,v) \\
	&=\tilde\Omega_0\tilde v+\tilde\Omega_1\tilde v+\tilde\Omega_2 P^N v+\ql(QP^N v)+T(QP^T,v)+Qh, \end{align*}
	with $\tilde\Omega_0:=Q\Omega_0\uno_{\R\setminus B(x_0,r)}Q^{-1}$, $\tilde\Omega_1:=Q\pa{\Omega_1-\mz T^*(Q^{-1},Q)}Q^{-1}-T^*(Q,P^N)Q^{-1}$ and $\tilde\Omega_2:=\ql Q+2Q\ql P^N$. Notice that $\tilde\Omega_0,\tilde\Omega_2\in L^2(\R,\R^{m\times m})$ and $\tilde\Omega_1\in L^{2,1}(\R,\R^{m\times m})$.
	Recall that $P^N v\in L^2(\R)$ by Lemma \ref{PNprop}
	(see also \cite[Lemma~3.5]{DLS17} for related properties).
	
	\underline{Step 2.} Next, we use the last equation satisfied by $\tilde v$ in order to estimate locally the $L^{2,\infty}$-norm of $v$. As $\tilde v\in L^2(\R,\R^m)$, we have
	\begin{align*} \tilde v=\mql\ql\tilde v&=\mql(\tilde\Omega_0\tilde v)+\mql(\tilde\Omega_1\tilde v)+\mql(\tilde\Omega_2 P^N v) \\
	&\quad+QP^N v+\mql T(QP^T,v)+\mql(Qh). \end{align*}
	Fix any radius $r\le\frac{R}{2}$ and an integer $s\ge 1$.
	
	Notice that 
	$\mql(\tilde\Omega_0\tilde v)=c\abs{x}^{-1/2}*(\tilde\Omega_0\tilde v)$ restricts to a $L^\infty$ function on $B=B(x_0,r)$, as $\tilde\Omega_0\tilde v$ is supported far from $B$, while $\mql(Qh)\in L^\infty(\R)$
	since $\abs{x}^{-1/2}\in L^1(\R)+L^\infty(\R)$ and $h\in L^1\cap L^\infty(\R)$. Moreover, being $\abs{x}^{-1/2}\in L^{2,\infty}(\R)$,
	\begin{align*} &\bnorm{\mql(\tilde\Omega_1\tilde v)}_{L^{2,\infty}(B)}
	\ls\bnorm{\tilde\Omega_1\tilde v}_{L^1(B(x_0,2^sr))}+\sum_{j=s}^\infty r^{1/2}\norm{\abs{x}^{-1/2}*(\tilde\Omega_1\tilde v\uno_{A_j})}_{L^\infty(B)} \\
	&\ls\bnorm{\tilde\Omega_1}_{L^{2,1}(B(x_0,2^sr))}\norm{v}_{L^{2,\infty}(B(x_0,2^sr))}+\sum_{j=s}^\infty 2^{-j/2}\bnorm{\tilde\Omega_1\tilde v}_{L^1(A_j)} \\
	&\ls\pa{\norm{\Omega_1}_{L^{2,1}(B(x_0,2^sr))}+\norm{\Omega_0}_{L^2(B(x_0,R))}}\norm{v}_{L^{2,\infty}(B(x_0,2^sr))}+\sum_{j=s}^\infty 2^{-j/2}\norm{v}_{L^{2,\infty}(A_j)}, \end{align*}
	where we used Theorem \ref{tustarest} and we neglected $\bnorm{\tilde\Omega_1}_{L^{2,1}}$ in the estimate of $\bnorm{\tilde\Omega_1\tilde v}_{L^1(A_j)}$, as well as $\norm{P^N}_\hmz$ (recall that $A_j=B(x_0,2^{j+1}r)\setminus B(x_0,2^j r)$). Similarly, by Lemmas \ref{PNprop} and \ref{loctstar},
	\begin{align*} &\bnorm{\mql(\tilde\Omega_2P^N v)+QP^N v}_{L^{2,\infty}(B)}\ls\sum_{j=0}^\infty 2^{-j/2}\norm{P^N v}_{L^2(B(x_0,2^jr))} \\
	&\ls\sum_{j=0}^\infty 2^{-j/2}\pa{\norm{w}_{\hmz(B(x_0,2^{j+2}r))}+\norm{v}_{L^{2,\infty}(B(x_0,2^{j+2}r))}}\norm{v}_{L^{2,\infty}(B(x_0,2^{j+1}r))} \\
	&\quad +\sum_{j=0}^\infty\sum_{\ell=j+1}^\infty 2^{-j/2-(\ell-j)/4} \pa{\norm{w}_{\hmz(B(x_0,2^{j+2}r))}+\norm{v}_{L^{2,\infty}(A_\ell)}}\norm{v}_{L^{2,\infty}(A_\ell)} \\
	&\ls\pa{\norm{w}_{\hmz(B(x_0,2^{s+1}r))}+\norm{v}_{L^{2,\infty}(B(x_0,2^{s+1}r))}}\norm{v}_{L^{2,\infty}(B(x_0,r))} \\
	&\quad +\sum_{j=s+1}^\infty 2^{-j/4}\norm{v}_{L^{2,\infty}(B(x_0,2^jr))}, \end{align*}
	where we neglected $\bnorm{\tilde\Omega_2}_{L^2}$ and $\norm{v}_{L^{2,\infty}},\norm{w}_\hmz\ls\norm{v}_{L^2}$. Finally, using Lemma \ref{loct} and neglecting $\norm{QP^T}_\hmz$,
	\begin{align*} &\bnorm{\mql T(QP^T,v)}_{L^2(B)} \\
	&\ls\norm{v}_{L^{2,\infty}(B(0,2^s))}\sum_{j=s}^\infty 2^{s/2-j/4}\norm{QP^T}_{\hmz(B(0,2^jr))}+\sum_{j=s}^\infty 2^{-j/4}\norm{v}_{L^{2,\infty}(A_j)} \\
	&\ls2^{s/2}\Bigg(\norm{\Omega_0}_{L^2(B(x_0,R))}+\sum_{j=s}^\infty 2^{-j/4}\norm{P^T}_{\hmz(B(x_0,2^jr))}\Bigg)\norm{v}_{L^{2,\infty}(B(x_0,2^sr))} \\
	&\quad+\sum_{j=s}^\infty 2^{-j/4}\norm{v}_{L^{2,\infty}(A_j)}. \end{align*}
	Combining the previous estimates, given $\epsilon$ we can fix $R$ (depending on $\epsilon$ and $s$) so small that
	\[ \norm{v}_{L^{2,\infty}(B(x_0,r))}\le \epsilon\norm{v}_{L^{2,\infty}(B(x_0,2^sr))}+C\sum_{j=s+1}^\infty 2^{-j/4}\norm{v}_{L^{2,\infty}(B(x_0,2^jr))}+Cr^{1/2} \]
	for all $r\le\frac{R}{2}$, with $C$ independent of $\epsilon$ and $s$. Choosing $s$ large enough, it follows that
	\begin{equation} \label{mor2infty} \norm{v}_{L^{2,\infty}(B(x_0,r))}\ls r^\beta. \end{equation}
	for some $0<\beta<\frac14$ and all $r>0$ small enough (see e.g. \cite[Lemma~A.8]{BRS16}, applied to the sequence $b_0:=\norm{v}_{L^{2,\infty}}$, $b_k:=\norm{v}_{L^{2,\infty}(B(x_0,2^{-k}R))}$ for $k>0$). Hence, being $\norm{v}_{L^2}$ finite, this holds for all $r>0$.	
	Notice that this inequality is uniform in $x_0$.
	
	\underline{Step 3.} We define $\zeta:=\mql(\tilde\Omega_1\tilde v)+\mql(\tilde\Omega_2 P^N  v)\in L^{2,\infty}(\R,\R^m)$ (where $\tilde\Omega_1$ and $\tilde\Omega_2$, defined above, depend on $x_0$). From \eqref{mor2infty} and the preceding estimates we deduce
	\[ \bnorm{\ql\zeta}_{L^1(B(x,r))}\ls r^\beta \]
	for all $r>0$ and all $x\in\R$.
	This Morrey-type estimate for the local $L^1$-norm of $\ql\zeta$ implies that $\zeta\in L^p_{loc}(\R,\R^m)$ for some $2<p<\infty$: indeed, arguing as in \cite[Proposition~3.1]{Ada75}, we have $\zeta=c\abs{x}^{-1/2}*\ql\zeta$ and thus, for a.e. $x\in\R$ and all $r>0$,
	\begin{align*} \abs{\zeta(x)}&\ls\sum_{j\in\Z}(2^jr)^{-1/2}\bnorm{\ql\zeta}_{L^1(B(x,2^jr)\setminus B(x,2^{j-1}r))} \\
	&\ls M(\ql\zeta)(x)\sum_{j\le 0}2^{j/2}r^{1/2}+\sum_{j>0}(2^jr)^{\beta-1/2} \\
	&\ls r^{1/2}M(\ql\zeta)(x)+r^{\beta-1/2}. \end{align*}
	Optimizing this inequality in $r$, we infer that
	\[ \abs{\zeta(x)}\ls M(\ql\zeta)(x)^{(\mz-\beta)/(1-\beta)}. \]
	for all $x\in\R$. The right-hand side lies in $L^{(1-\beta)/(\mz-\beta),\infty}(\R)$ (as $\ql\zeta\in L^1(\R,\R^m)$), so we get the claim for any $2<p<\frac{1-\beta}{\mz-\beta}$. In particular, we get $\norm{\zeta}_{L^2(B(x,r))}\ls r^{\beta'}$ for some $\beta'>0$. On the other hand,
	\[ \tilde v-\zeta=\mql(\tilde\Omega_0\tilde v)+QP^N v+\mql T(QP^T,v)+\mql(Qh) \]
	and so the estimates derived in Step 2 give $\norm{\tilde v-\zeta}_{L^2(B(x_0,r))}\ls r^{\beta}$, for all $0<r<\frac{R}{2}$. We deduce that, for all $0<r<\frac{R}{2}$,
	\[ \norm{v}_{L^2(B(x_0,r))}\le\norm{\tilde v-\zeta}_{L^2(B(x_0,r))}+\norm{\zeta}_{L^2(B(x_0,r))}\ls r^\alpha, \]
	with $\alpha:=\min\set{\beta,\beta'}$. Hence $\norm{v}_{L^2(B(x_0,r))}\ls r^\alpha$ for all $r>0$, uniformly in $x_0$.
	
	\underline{Step 4.} Finally, Lemma \ref{circletoline} gives the two identities
	\[ \ql(P^T v)=h-\ql P^T v+T(P^T,v), \]
	\[ \rz\ql(P^N v)=\rz\ql P^N v-U(P^N,v), \]
	as $\rz\ml w=-\nabla w$. Arguing as in the proof of Lemma \ref{loct}, but using Theorem \ref{tuest} in place of Corollary \ref{tuinfty}, we finally get
	\[ \norm{T(P^T,v)}_{L^1(B)}\ls\norm{v}_{L^2(A_0)}+\sum_{j=1}^\infty\norm{T(P^T,v\uno_{A_j})}_{L^1(B)}
	\ls\sum_{j=0}^\infty2^{-j/2}\norm{v}_{L^2(A^j)}\ls r^\alpha \]
	and similarly $\norm{U(P^N,v)}_{L^1(B)}\ls r^\alpha$. The thesis follows.
\end{proof}

\begin{corollary}\label{lpreg}
	We have $v\in L^p_{loc}(\R,\R^m)$ and $w\in C^{0,\gamma}_{loc}(\R,\R^m)$, for some $p>2$ and some $\gamma>0$.\end{corollary}
 
We include the standard proof for the reader's convenience.

\begin{proof}
	Arguing as in Step 3 of the proof of Theorem \ref{regw}, we infer that
	\[ \int_{B(x_0,4)}\abs{P^T v}^p+\int_{B(x_0,4)}\abs{\rz(P^N v)}^p\ls 1 \]
	for some $p>2$, uniformly in $x_0$. If $\rho\in C^\infty_c(B(x_0,4))$ is a cut-off function with $\rho=1$ on $B(x_0,2)$,
	\[ P^N v=-\rz\rz(P^N v)=-\rz(\rho\rz(P^N v))-\rz((1-\rho)\rz(P^N v)). \]
	Using \cite[Proposition~2.4.8]{Gra14C} applied to $-i\operatorname{sgn}(\xi)$ (whose inverse Fourier transform is $(-1)$-homogeneous) and the fact that $(1-\rho)\rz(P^N v)\in L^2(\R,\R^m)$ is supported far from $B(x_0,1)$,
	\[ \norm{\rz((1-\rho)\rz(P^N v))}_{L^\infty(B(x_0,1))}\ls 1 \]
	and, from the $L^p$ boundedness of the Hilbert--Riesz transform,
	\[ \norm{\rz(\rho\rz(P^N v))}_{L^p}\ls\norm{v}_{L^p(B(x_0,4))}. \]
	We deduce that $v=P^T v+P^N v$ also satisfies an estimate $\int_{B(x_0,1)}\abs{v}^p\ls 1$, uniformly in $x_0$. In particular, $\norm{\ql w}_{L^2(B(x_0,r))}\ls r^\gamma$ with $\gamma=\mz-\frac{1}{p}\in\pa{0,\mz}$ (for $0<r<1$ and hence for all $r>0$).
	Using Lemma \ref{tech2infty} we deduce that
	\[ \pa{\media_{B(x_0,r)}\abs{w-(w)_{B(x_0,r)}}^2}^{1/2}\ls\sum_{j=0}^\infty 2^{-j/2}(2^j r)^\gamma\ls r^\gamma. \]
	This is the integral characterization of H\"older continuity with exponent $\gamma$: see e.g. \cite[Theorem~III.1.2]{Gia83}.
\end{proof}

Applying a rotation before taking the stereographic projection, we arrive at the following.

\begin{corollary}
	The map $u\circ\psi_\ell^{-1}:\mathcal S^1\to\R^m$ is H\"older continuous and, being $\ell$ is arbitrary, $u$ is H\"older continuous.
\end{corollary} 

\section{Higher regularity of \texorpdfstring{$\mz$}{1/2}-harmonic maps}\label{higherreg}

In this section we prove that $\mz$-harmonic maps $u\in\hmz(\de S,\N)$ are $C^{k-1,\delta}_{loc}$, for any $0<\delta<1$, whenever $\N$ is a $C^k$-closed manifold ($k\ge 2$). We mention that higher regularity of the so-called {\em half-wave} maps into $S^2$ has recently been obtained in \cite{LS17}.

Throughout the section, we will say that $a\in\mathcal{S}'(\R)$ belongs to $H^{s,p}_{loc}(\R)$ (with $s\ge 0$, $1<p<\infty$) if $\psi a\in H^{s,p}(\R)$ for any $\psi\in C^\infty_c(\R)$.

Corollary \ref{lpreg} shows that $\ql w\in L^p_{loc}(\R,\R^m)$, for $w:=u\circ\psi_\ell^{-1}\circ\Pi^{-1}$, for some $p>2$. We now bootstrap this information to higher regularity.
We first prove two results concerning the regularity of the commutator $\rz(ab)-a\rz(b)$. The proofs will rely on the technique of splitting products into {\em paraproducts}, using
the Littlewood--Paley decomposition (see Section \ref{parasec}):
\[ ab=\sum_j a_j b^{j-3}+\sum_j a^{j-3} b_j+\sum_{\abs{j-k}\le 2} a_j b_k, \qquad\widehat{a}_j=\varrho_j\widehat a,\quad \widehat{b}_j=\varrho_j\widehat b. \]
We will treat the first and the third summands together, namely we will just decompose
\[ ab=\sum_j a_j b^{j+2}+\sum_j a^{j-3} b_j. \]

\begin{lemmaen}\label{commstart}
	Let $a\in\hmzb(\R)$ with $\ql a\in L^p(\R)$, for some $2<p<\infty$, and $b\in\dot H^{-1/2}(\R)$. Then
	\[ \norm{\rz(ab)-a\rz(b)}_{L^{2p/(p+2)}}\ls\bnorm{\ql a}_{L^p}\norm{b}_\hmmz. \]
\end{lemmaen}

Notice that $\rz(ab)$ is defined, as for $\rho\in\mathcal{S}(\R)$ it holds $\mathcal{F}^{-1}\bra{-i\frac{\xi}{(\epsilon^2+\abs{\xi}^2)^{1/2}}\hat\rho}\to\rz\rho$ in $\hmzb(\R)$ and $ab$ extends to a continuous functional on this space (see Appendix \ref{functapp}).

\begin{proof}
	Notice that the commutator vanishes if $a$ is constant. Thus, as the proof of Lemma \ref{hmzapprox} and Remark \ref{linftyapprox} show, we can assume $\widehat{a},\widehat{b}\in C^\infty_c(\R\setminus\set{0})$.\footnote{We can assume $\widehat a$ has compact support in $\R\setminus\set{0}$, by replacing it with $\widecheck w_k$ (defined as in Lemma \ref{hmzapprox}): the norm $\bnorm{\ql a}_{L^q}$ stays controlled by Lemma \ref{hmztol2} and the same argument of Remark \ref{linftyapprox}; we can then choose $\abs{\xi}^{1/2}\hat v_k$ arbitrarily close to $\abs{\xi}^{1/2}w_k$ in $L^{p'}(\R)$, obtaining $\ql v_k$ close to $\ql\widecheck w_k$ in $L^p(\R)$.} Using the homogeneous decomposition, we write
	\[ \rz(ab)-a\rz(b)=\sum_{j\in\Z}(\rz(a^{j-3}b_j)-a^{j-3}\rz(b_j))+\sum_{j\in\Z}(\rz(a_jb^{j+2})-a_j\rz(b^{j+2})) \]
	and notice that the first sum vanishes since
	\begin{align*} \mathcal{F}(\rz(a^{j-3}b_j)-a^{j-3}\rz(b_j))(\xi)=&-i\sgn(\xi)\int\widehat{a^{j-3}}(\xi-\eta)\widehat{b_j}(\eta)\,d\eta \\
	&+i\int\widehat{a^{j-3}}(\xi-\eta)\sgn(\eta)\widehat{b_j}(\eta)\,d\eta=0 \end{align*}
	(as $\sgn(\eta)=\sgn(\xi)$ whenever $\widehat{a^{j-3}}(\xi-\eta)\widehat{b_j}(\eta)\neq 0$).
	
	Since $\rz$ is an isomophism of $L^{2p/(p+2)}(\R)$ and of $\hmmz(\R)$, it suffices to bound $\sum_{j\in\Z}a_jb^{j+2}$ in $L^{2p/(p+2)}(\R)$. We do this by duality: let $h\in\mathcal{S}(\R)$ and observe that
	\[ \int\sum_{j\in\Z}a_jb^{j+2}h=\int\sum_{j\in\Z}a_jb^{j+2}h^{j+5}\ls\int\Bigg(\sum_{j\in\Z}2^j\abs{a_j}^2\Bigg)^{1/2}\Bigg(\sum_{j\in\Z}2^{-j}\abs{b^{j+2}}^2\Bigg)^{1/2}(Mh), \]
	as $\mathcal{F}(a_jb^{j+2})$ is supported in $B(0,2^{j+4})$ and as we have the elementary inequality $\abs{h^{j+5}}\ls Mh$.
	Note that $\norm{Mh}_{L^{2p/(p-2)}}\ls\norm{h}_{L^{2p/(p-2)}}$, while the $\ell^2(\Z)$-norm $\Big(\sum_{j\in\Z}\abs{2^{-j/2}b^{j+2}}^2\Big)^{1/2}$ equals
	\[ \Bigg(\sum_{j\in\Z}\abs{\sum_{k=-\infty}^2 2^{-j/2}b_{j+k}}^2\Bigg)^{1/2}
	\le\sum_{k=-\infty}^2\Bigg(\sum_{j\in\Z}\abs{2^{-j/2}b_{j+k}}^2\Bigg)^{1/2}
	\ls\Bigg(\sum_{j\in\Z}2^{-j}\abs{b_j}^2\Bigg)^{1/2} \]
	(as $\sum_{j\in\Z}\abs{2^{-j/2}b_{j+k}}^2=2^k\sum_{j\in\Z}2^{-j}\abs{b_j}^2$), so that, by Plancherel's identity,
	\[ \norm{\Bigg(\sum_{j\in\Z}2^{-j}\abs{b^{j+2}}^2\Bigg)^{1/2}}_{L^2}^2\ls\sum_{j\in\Z}2^{-j}\int\abs{b_j}^2
	=\int\sum_{j\in\Z}2^{-j}\rho_j^2\abs{\widehat{b}}^2\ls\int\abs{\xi}^{-1}\abs{\widehat{b}(\xi)}^2
	=\norm{b}_{\hmmz}^2. \]
	To conclude, using \cite[Theorem~1.6.3]{Tri83} with the multipliers $2^{j/2}\abs{\xi}^{-1/2}(\varrho_{j-1}+\varrho_j+\varrho_{j+1})$ and \cite[Theorem~6.1.2]{Gra14C}, we infer
	\[ \norm{\pa{\sum_{j\in\Z}2^j\abs{a_j}^2}^{1/2}}_{L^p}\ls\norm{\pa{\sum_{j\in\Z}\abs{\ql a_j}^2}^{1/2}}_{L^p}
	\ls\bnorm{\ql a}_{L^p}. \]
	To sum up, by H\"older's inequality we get the desired bound
	\[ \int\sum_{j\in\Z}a_jb^{j+2}h\ls\bnorm{\ql a}_{L^p}\norm{b}_\hmmz\norm{h}_{L^{2p/(p-2)}}. \qedhere \]
\end{proof}

\begin{lemmaen}\label{commhilb}
	Let $a\in H^{s,p}(\R)$ and $b\in L^q(\R)$, with $s>\frac{1}{p}$, $1<p,q<\infty$.
	Then, for any $\gamma>\frac{1}{p}$,
	\[ \norm{\rz(ab)-a\rz(b)}_{H^{s-\gamma,q}}\ls\norm{a}_{H^{s,p}}\norm{b}_{L^q}. \]
\end{lemmaen}

\begin{proof}
	We can assume $\widehat{a},\widehat{b}\in C^\infty_c(\R)$.
	We use the inhomogeneous Littlewood--Paley decomposition, so that $a=\sum_{j\ge 0} a_j$ and $b=\sum_{j\ge 0} b_j$,
	where $ \hat{a}_j=\varrho_j\hat{a}$.
	
	As in the previous proof, we need only estimate $\norm{\sum_{j\ge 0}a_jb^{j+2}}_{H^{s-\gamma,q}}$, as $\rz$ is an isomorphism of $H^{s-\gamma,q}(\R)$ and of $L^q(\R)$.	
	We have $\norm{a_j}_{L^\infty}\ls 2^{-j(s-1/p)}\norm{a}_{H^{s,p}}$ (see the proof of Corollary \ref{hspholder}).
	Given $h\in\mathcal{S}(\R)$, observe that $\mathcal{F}(a_j b^{j+2})$ vanishes outside $B(0,2^{j+4})$, so
	\[ \abs{\int\sum_{j\ge 0}a_jb^{j+2}h}
	=\abs{\sum_{j\ge 0}\int a_jb^{j+2}h^{j+5}}
	\ls\norm{a}_{H^{s,p}}\norm{Mb}_{L^q}\norm{h}_{F_{q',1}^{-(s-1/p)}}, \]
	thanks to the pointwise inequalities $\abs{b^{j+2}}\ls Mb$ ($Mb$ being the maximal function of $b$) and
	\[ \abs{\sum_{j\ge 0}a_j b^{j+2}h^{j+5}}\ls\sum_{j\ge 0}2^{-j(s-1/p)}\norm{a}_{H^{s,p}}(Mb)\abs{h^{j+5}}\ls \norm{a}_{H^{s,p}}(Mb)\sum_{j\ge 0}2^{-j(s-1/p)}\abs{h_j}. \]
	But $H^{-(s-\gamma),q'}=F_{q',2}^{-(s-\gamma)}\hookrightarrow F_{q',1}^{-(s-1/p)}$ (see \cite[Proposition~2.3.2/2]{Tri83}),
	so
	\[ \abs{\int\sum_{j\ge 0}a_jb^{j+2}h}
	\ls\norm{a}_{H^{s,p}}\norm{b}_{L^q}\norm{h}_{H^{-(s-\gamma),q'}}. \qedhere \]
\end{proof}

 We will implicitly use many times the following result.

\begin{lemmaen}\label{lemmapt}
	If $u\in H^{s,p}_{loc}(\R)$ for some $s\ge 1$ and $1<p<\infty$, then $P^T(u)\in H^{\min(s,k-1),p}_{loc}(\R)$.
\end{lemmaen}

\begin{proof}
	We can assume that $1\le s\le k$. The thesis is trivial for $s=1$, while for $s>1$ it follows from \cite[Theorem~5.2]{BM01}, the map $P^T$ being $C^{k-1}$-smooth. Notice that $u\in W^{1,sp}_{loc}(\R)$ by \cite[Lemma~3.1]{BM01} with $(p,q,s):=(sp,2,1)$, $(p_1,q_1,s_1):=(p,2,s)$, $(p_2,q_2,s_2):=(\infty,\infty,0)$ and the fact that $u\in H^{1,p}_{loc}(\R)\subseteq L^\infty_{loc}(\R)$.
\end{proof}

%

We need the following lemmata.
We use the dyadic partition of unity $(\varrho_j)_{j=0}^\infty\subseteq C^\infty_c(\R)$ introduced in Appendix \ref{functapp}.

\begin{lemmaen}\label{hmzintrep}
	If $f\in\dot H^{1/2}(\R)$ and $\rho\in C^\infty_c(B(0,1))$, we have
	\begin{equation}\label{ident} \big\langle(-\Delta)^{1/2}f,\rho\big\rangle=\sum_{j=0}^\infty\int((-\Delta)^{1/2}\rho)(\varrho_j f). \end{equation}
\end{lemmaen}

\begin{proof}
	Notice that $f\in L^1_{loc}(\R)$ and $(-\Delta)^{1/2}\rho\in L^\infty(\R)$, so each term in the right-hand side makes sense. By the remark after Lemma \ref{hmztol2}, the left-hand side equals $2\pi\int\abs{\xi}\widehat{f}(\xi)\obar{\widehat\rho(\xi)}\,d\xi$.
	
	For any $j\ge 2$, by Lemma \ref{hmzannest} and the fact that $\mathcal{F}^{-1}(\abs{\xi})\in C^\infty(\R\setminus\set{0})$ is homogeneous of degree $-2$ (see \cite[Proposition~2.4.8]{Gra14C}),
	\[ \abs{\int((-\Delta)^{1/2}\rho)(\varrho_j f)}\ls 2^{-2j}\norm{f}_{L^1(B(0,2^{j+1})\setminus B(0,2^{j-1}))}\ls 2^{-2j}\norm{f}_{L^1(B(0,1))}+2^{-j}(j+1)\norm{f}_{\dot H^{1/2}}. \]
	Therefore  the sum in the right-hand side of \eqref{ident} converges and
	is bounded by $\norm{f}_{L^1(B(0,1))}+\norm{f}_{\dot H^{1/2}}$.
	Hence, by Lemma \ref{hmzapprox}, it is enough to prove the identity \eqref{ident} on $\mathcal{S}(\R)+\R$.
	
	If $f\in\mathcal{S}(\R)$, the identity is trivially satisfies since in this case we have
 	\[ \sum_{j=0}^\infty\int((-\Delta)^{1/2}\rho)(\varrho_j f)=\int((-\Delta)^{1/2}\rho)f=2\pi\int\abs{\xi}\widehat{f}(\xi)\obar{\widehat{\rho}(\xi)}\,d\xi. \]
	If $f=c$ is constant then
	\[ \big\langle(-\Delta)^{1/2}c,\rho\big\rangle=0=2\pi c\lim_{N\to\infty}\int\abs{\xi}\widehat{\rho}(\xi)\sum_{0}^N\obar{\widehat{\varrho_j}(\xi)}\,d\xi=\sum_{j=0}^\infty\int((-\Delta)^{1/2}\rho)(\varrho_j c), \]
	the second equality being true since $\sum_{0}^N\obar{\widehat{\varrho_j}(\xi)}$ approximates the Dirac mass $\delta_0$ as $N\to\infty$.
\end{proof}

 \begin{lemmaen}\label{lemmaw}
	Assume $w\in\dot H^{1/2}(\R)$ is supported outside $B(x_0,2)$, for some $x_0\in\R$. Then the distribution $(-\Delta)^{1/2}w$ restricts to a $C^\infty$ function on $B(x_0,1)$.
\end{lemmaen}

\begin{proof}
	We can assume $x_0=0$. For $\rho\in C^\infty_c(B(0,1))$ and $k\ge 0$ integer,
	Lemma \ref{hmzintrep} gives
	\begin{align*} &\big\langle(-\Delta)^{1/2}w;(-\Delta)^k\rho\big\rangle=\sum_{j\ge 1}\int ((-\Delta)^{k+1/2}\rho)\cdot(\varrho_j w) \\
	&\ls\sum_{j\ge 1}2^{-(2k+2)j}\norm{\rho}_{L^1}\norm{w}_{L^1(B(0,2^{j+1})\setminus B(0,2^{j-1}))}
	\ls\sum_{j\ge 1}2^{-(2k+2)j}\norm{\rho}_{L^2}\cdot (j+1)2^j\ls\norm{\rho}_{L^2}, \end{align*}
	where the inequalities follow from \cite[Proposition~2.4.8]{Gra14C} and Lemma \ref{hmzannest}.
	So, calling $f$ the restriction of $(-\Delta)^{1/2}w$ to $B(0,1)$, we have $(-\Delta)^k f\in L^2(B(0,1))$. Equivalently, $\frac{d^{2k}f}{dx^{2k}}\in L^2(B(0,1))$. This implies that $\frac{d^{2k-1}f}{dx^{2k-1}}\in C^0(B(0,1))$ for all $k\ge 0$, hence $f\in C^\infty(B(0,1))$.
\end{proof}
 
\begin{proof}[Proof of Theorem \ref{regul}.]
We fix $x_0\in\R$ and we take  a cut-off function $\eta\in C^\infty_c(B_1(x_0))$ satisfying $\eta=1$ in a neighborhood of $x_0$. Recall from Lemma \ref{circletoline} that
\[ P^T(w)\ml w=h,\quad P^N(w)\nabla w=0, \]
with $h=-\frac{2}{1+x^2}P^T(w)\pa{R_\ell((f_j)_{j=1}^k)\circ\Pi^{-1}}\in L^1\cap L^{\infty}(\R)$.
Therefore we have
\begin{align}\label{eq:commeasy}
	\begin{aligned}\eta\nabla w=&\eta P^T(w)\nabla w=-\eta P^T(w)\rz(-\Delta)^{1/2}w \\
=&\rz(\eta P^T(w)(-\Delta)^{1/2}w)-\eta P^T(w)\rz(-\Delta)^{1/2}w-\rz(\eta h).\end{aligned}
\end{align}
We remark that $w\in H^{1/2,p}_{loc}(\R,\R^m)$: by Lemma \ref{qlintrep}
\[ \babs{\ql(\psi w)-\psi\ql w}(x)\ls\int\frac{\abs{\psi(x)-\psi(y)}}{\abs{x-y}^{3/2}}\abs{w(y)}\,dy\ls g*\abs{w}(x)\in L^\infty(\R) \]
with $g(x):=\min(\abs{x}^{1/2},\abs{x}^{-3/2})\in L^1(\R)$, for any $\psi\in C^\infty_c(\R)$. Hence
$\ql(\psi w)$ lies both in $L^2(\R)$ and in $L^p(\R)+L^\infty(\R)$ and thus it lies in $L^p(\R)$, as well. Since trivially $\psi w\in L^p(\R)$, \cite[Theorem~1.6.3]{Tri83} gives $\psi w\in F_{p,2}^{1/2}(\R)=H^{1/2,p}(\R)$.

Thus $\eta P^T(w)\in H^{1/2,p}(\R)$ and, using again \cite[Theorem~1.6.3]{Tri83} (with multipliers $\abs{\xi}^{1/2}(1+\abs{\xi}^2)^{-1/4}(\varrho_{j-1}+\varrho_j+\varrho_{j+1})$ for $j\in\Z$) and \cite[Theorem~6.1.2]{Gra14C}, we infer that $\eta P^T(w)$ and $\ml w$ satisfy the hypotheses of Lemma \ref{commstart}. So, in view of \eqref{eq:commeasy}, we get
$\eta\nabla w\in L^{2p/(p+2)}(\R)$, i.e. $w\in H^{1,\tilde p}_{loc}(\R)$ with $\tilde p=2p/(p+2)$.

We now fix another cut-off function $\phi\in C^\infty_c(\R)$ such that $\phi=1$ on $B(x_0,2)$ and we set
\[ w_1:=\phi w, \quad w_2:=(1-\phi)w. \]
Lemma \ref{lemmaw} yields that $(-\Delta)^{1/2}w_2\in C^{\infty}(B(x_0,1))$.
Now assume that we already know $w\in H^{s,\tilde p}_{loc}(\R)$ for some real $s\ge 1$:
by Lemma \ref{lemmapt} we get
$h\in H^{\min(s,k-1),\tilde p}_{loc}(\R)$, so $\tilde h:=P^T(w)\ml w_1=-P^T(w)\ml w_2+h$ restricts to a function in $H^{\min(s,k-1)}_{loc}(B(x_0,1))$. We rewrite \eqref{eq:commeasy} as
\begin{align*} \eta\nabla w=&\eta P^T(w)\nabla w_1=-\eta P^T(w)\rz\ml w_1 \\
=&\rz(\eta P^T(w)\ml w_1)-\eta P^T(w)\rz\ml w_1-\rz(\eta\tilde h). \end{align*}
The commutator on the right-hand side belongs to $H^{\min(s,k-1)-\gamma,\tilde p}(\R)$, for any $\gamma>\frac{1}{\tilde p}$, thanks to Lemma \ref{commhilb} (applied with $p=q:=\tilde p$). Therefore $\eta\nabla w\in H^{\min(s,k-1)-\gamma}(\R)$, which implies
$w\in H^{\min(s+1,k)-\gamma,\tilde p}_{loc}(\R)$. Iterating this procedure
we eventually get
\[ w\in\bigcap_{\gamma>1/\tilde p}H^{k-\gamma,\tilde p}_{loc}(\R). \]
We now show that, for any fixed $1<p<\infty$,
\[ w\in\bigcap_{\gamma>1/p} H^{k-\gamma,p}_{loc}(\R). \]
Since $k\ge 2$, we know that $w\in H^{1,q}_{loc}(\R)$ for all $q<\frac{\tilde p}{2-\tilde p}$ (because $H^{2-\gamma,\tilde p}_{loc}(\R)\subseteq H^{1,q}_{loc}(\R)$ with $\frac{1}{q}=\gamma+\frac{1}{\tilde p}-1$ whenever $\gamma<1$, see \cite[Theorem~2.7.1]{Tri83}).
Proceeding as above we obtain
\[ w\in\bigcap_{\gamma>1/q}H^{k-\gamma,q}_{loc}(\R) \]
for all $q<\frac{\tilde p}{2-\tilde p}$ (notice that $\frac{\tilde p}{2-\tilde p}>\tilde p$).
Iterating this with $q$ in place of $\tilde p$, we will eventually reach an exponent $q\ge 2$ and hence, as $\bigcap_{\gamma>1/q}H^{2-\gamma,q}_{loc}(\R)\subseteq\bigcap_{1<r<\infty}H^{1,r}_{loc}(\R)$, all exponents in $(1,\infty)$. This proves the assertion. Finally, applying Corollary \ref{hspholder},
\[ w\in\bigcap_{1<p<\infty}\bigcap_{\epsilon>0}H^{k-1/p-\epsilon,p}_{loc}(\R)\subseteq\bigcap_{0<\delta<1} C^{k-1,\delta}_{loc}(\R). \]
So $u\in C^{k-1,\delta}(\de S)$ for all $0<\delta<1$. In particular, if $\N$ is $C^\infty$-smooth then $u$ is $C^\infty$ as well.
\end{proof}

\nocite{*}
\printbibliography

\end{document}